\documentclass[10pt]{amsart}

\usepackage{hyperref}
\usepackage{enumerate}
\usepackage{comment}

\makeatletter

\@namedef{subjclassname@2010}{

  \textup{2010} Mathematics Subject Classification}

\usepackage{amssymb, amsmath,amsthm}
\usepackage{mathrsfs}
\newtheorem{thm}{Theorem}[section]
\newtheorem{thmMain}{Theorem}
\newtheorem*{thm1}{Theorem}
\newtheorem{prop}[thm]{Proposition}

\newtheorem{ques}{Question}
\newtheorem{cor}[thm]{Corollary}
\newtheorem{corMain}[thmMain]{Corollary}
\newtheorem{lem}[thm]{Lemma}
\theoremstyle{definition}
\newtheorem{rem}{Remark}
\newtheorem{def1}{Definition}

\newcommand{\ra}{\rightarrow}
\newcommand{\bk}{\backslash}
\newcommand{\mc}{\mathcal}
\newcommand{\mf}{\mathfrak}
\newcommand{\mb}{\mathbb}
\newcommand{\sg}{\sigma}

\renewcommand{\ss}{\substack}
\newcommand{\llf}{\left\lfloor}

\newcommand{\e}{\varepsilon}
\newcommand{\rrf}{\right\rfloor}

\renewcommand{\bar}{\overline}

\begin{document}

\title{Large Sums of High Order Characters}
\author{Alexander P. Mangerel}
\address{Department of Mathematical Sciences, Durham University, Stockton Road, Durham, DH1 3LE, UK}
\email{smangerel@gmail.com}
\maketitle

\begin{abstract}
Let $\chi$ be a primitive character modulo a prime $q$, and let $\delta > 0$. It has previously been observed that if $\chi$ has large order $d \geq d_0(\delta)$ then $\chi(n) \neq 1$ for some $n \leq q^{\delta}$, in analogy with Vinogradov's conjecture on quadratic non-residues. We give a new and simple proof of this fact. We show, furthermore, that if $d$ is squarefree then for \emph{any} $d$th root of unity $\alpha$ the number of $n \leq x$ such that $\chi(n) = \alpha$ is $o_{d \ra \infty}(x)$ whenever $x > q^\delta$. Consequently, when $\chi$ has sufficiently large order the sequence $(\chi(n))_{n \leq q^\delta}$ cannot cluster near $1$ for any $\delta > 0$.

Our proof relies on a second moment estimate for short sums of the characters $\chi^\ell$, averaged over $1 \leq \ell \leq d-1$, that is non-trivial whenever $d$ has no small prime factors. In particular, given any $\delta > 0$ we show that for all but $o(d)$ powers $1 \leq \ell \leq d-1$, the partial sums of $\chi^\ell$ exhibit cancellation in intervals $n \leq q^\delta$ as long as $d \geq d_0(\delta)$ is prime, going beyond Burgess' theorem. Our argument blends together results from pretentious number theory and additive combinatorics.

Finally, we show that, \emph{uniformly} over prime $3 \leq d \leq q-1$, the P\'{o}lya-Vinogradov inequality may be improved for $\chi^\ell$ on average over $1 \leq \ell \leq d-1$, extending work of Granville and Soundararajan.
\end{abstract}

\section{Introduction and Main Results}
\subsection{Background}
Understanding the value distribution of Dirichlet characters is a central theme in analytic number theory. An old and famous conjecture of I.M. Vinogradov predicts that the least quadratic non-residue $n_p$ modulo a prime $p$ satisfies $n_p \ll_{\delta} p^{\delta}$ as $p \ra \infty$ for any $\delta > 0$. 
This conjecture, relating to negative values of the Legendre symbol $\left(\tfrac{\cdot}{p}\right)$, may be generalized to other primitive Dirichlet characters. One can ask whether the least integer $n_{\chi}$ for which a primitive character $\chi \pmod{q}$ yields $\chi(n) \neq 0, 1$ satisfies $n_{\chi} \ll_{\delta}q^{\delta}$ as $q \ra \infty$ for any $\delta > 0$. \\
These problems may be recast in terms of cancellation in short character sums. For the Legendre symbol modulo $p$ it is a folklore conjecture going beyond Vinogradov's that
$$
\left|\sum_{n \leq x} \left(\frac{n}{p}\right)\right| = o_{p\ra \infty}(x) \text{ whenever $x > p^{\delta}$ for any $\delta > 0$,}
$$
and more generally for any primitive Dirichlet character $\chi \pmod{q}$ we expect that
\begin{equation}\label{eq:chiVino}
\left|\sum_{n \leq x} \chi(n) \right| = o_{q\ra \infty}(x) \text{ whenever $x > q^{\delta}$ for any $\delta > 0$.}
\end{equation}
Currently, the best general result towards \eqref{eq:chiVino}, due to Burgess \cite{Bur}, allows for such cancellation (at least for $q$ cube-free) whenever $\delta > 1/4$, improved to $\delta \geq 1/4$ for prime $q$ by Hildebrand \cite{Hil}. However, it is a notoriously difficult problem to extend the zero-free regions of Dirichlet $L$-functions for \emph{individual characters}. It is therefore desirable to determine other sufficient criteria that a primitive character might satisfy that guarantees cancellation in its partial sums in a range going beyond Burgess' theorem. \\
Both of the above conjectures are well-known to hold (in a much stronger form) under the assumption of the Generalized Riemann Hypothesis (GRH), but should even hold assuming far less. Indeed, Granville and Soundararajan \cite[Cor. 1.2]{GSBurg} showed that \eqref{eq:chiVino} holds as long as $L(s,\chi)$ has sufficiently few zeros in certain small rectangles near the line $\text{Re}(s) = 1$, a condition that is easily implied for \emph{typical} characters $\chi$ by classical zero density estimates. \\
As has been elaborated upon in various works (see e.g., \cite{BoGo}, \cite{FrGo}, \cite{Man} and \cite{GraMan}), there is also a relationship between such questions about short character sums and corresponding estimates for \emph{maximal} character sums, in particular regarding improvements to the P\'{o}lya-Vinogradov inequality, which asserts that any non-principal character modulo $q$ satisfies
\begin{equation}\label{eq:Mchi}
M(\chi) := \max_{1 \leq t \leq q} \left|\sum_{n \leq t} \chi(n)\right| \ll \sqrt{q}\log q.
\end{equation}
Montgomery and Vaughan \cite{MV} showed that, assuming GRH, the factor $\log q$ in \eqref{eq:Mchi} may be improved to $\log\log q$. This is sharp (up to the implicit constant) for quadratic characters according to a construction due to Paley \cite{Pal}. A general, \emph{unconditional} improvement to the P\'{o}lya-Vinogradov inequality has, however, resisted proof for over a century.\\
Characters of various fixed orders $d \geq 2$ have been considered in this connection in a number of works, see e.g. \cite{GSPret} and \cite{LamEven}.
Building on a breakthrough by Granville-Soundararajan \cite{GSPret}, Goldmakher \cite{Gold} showed that for characters with \emph{fixed odd} order $d$ the estimate \eqref{eq:Mchi} may be improved unconditionally 
 to 
\begin{equation}\label{eq:PVoddImp}
M(\chi) \ll \sqrt{q}(\log q)^{1-\delta(d) + o(1)}, \text{ with } \delta(d) = 1-\frac{d}{\pi} \sin(\pi/d) > 0
\end{equation}
(see \cite{LamMan} for more precise results). On the other hand, unconditional improvements to \eqref{eq:Mchi} for \emph{fixed even} order characters seem to provide a logjam to generally improving \eqref{eq:Mchi}. Once again, it would be valuable to determine other collections of primitive characters for which one may improve the P\'{o}lya-Vinogradov inequality. \\
In this paper, we will study the relationship between cancellation in short and maximal character sums and the size of the \emph{order} of a character $\chi \pmod{q}$, i.e, the minimal positive integer $d$ such that $\chi^d$ is principal.


\subsection{Motivating Questions}
In order to obtain improvements in estimates for short and maximal sums of a primitive character $\chi \pmod{q}$, we must rule out the heuristic  possibility that
$$
\chi(p) = 1 \text{ for all but very few } p \leq q^{\delta}.
$$ 
Since $\chi$ is multiplicative, this would further imply that
$$
\chi(n) = 1 \text{ for many integers } n \leq q^{\delta}.
$$ 
In order to preclude this possibility, therefore, 
we would like to \emph{quantify the frequency with which a character takes the value $1$} at integers $n$ up to some small threshold $x$. Note that this question is more refined than simply 
trying to bound $n_{\chi}$. \\
In this paper, we shall study the value distribution of characters $\chi$ of \emph{large} order (in a sense to be made precise shortly). If $\chi$ has modulus $q$ then the order $d$ of $\chi$ divides $\phi(q)$, and most such divisors grow as a function of $q$. Thus, the collection of such characters is rather substantial. Moreover, as the elementary results in Section \ref{sec:elem} below show, such characters do exhibit some variation in their values $\chi(n)$ for small $n$, which might suggest that their short and maximal sums also exhibit cancellation. \\
The role of a character's order in its value distribution has previously been considered in the work of K. K. Norton on small upper bounds for $n_{\chi}$ (see \cite{Nor} and especially \cite{Nor2}, which contains a substantial survey on the topic), and of Granville and Soundararajan \cite{GSUppL} on upper bounds for $|L(1,\chi)|$. To our knowledge, however, questions surrounding especially the \emph{paucity} of solutions to $\chi(n) = 1$ and \emph{distribution} of values $\chi(n) \neq 1$ for small $n$ have not appeared previously in the literature, in particular when the order of the character is allowed to grow as a function of its modulus.\\
In this direction we pose the following three motivating questions:
\begin{ques} \label{ques:fibres}
Fix $\delta \in (0,1)$. If $\chi$ is a primitive character modulo $q$ with order $d = d(q)$ growing with $q$, how few solutions $n \leq q^{\delta}$ are there to $\chi(n) = 1$, or more generally to $\chi(n) = \alpha$ when $\alpha^d = 1$?
\end{ques}
\begin{ques}\label{ques:Burg}
If $\chi$ is a character as in the previous question, can it be shown that for any fixed $\delta \in (0,1)$,
$$
\Big|\sum_{n \leq x} \chi(n) \Big| = o_{d \ra \infty}(x) \text{ for all } x > q^{\delta}?
$$ 
\end{ques}

\begin{ques}\label{ques:PV}
For a character $\chi$ as in the previous two questions, can it be shown that $M(\chi) = o_{d \ra \infty}(\sqrt{q}\log q)$?
\end{ques}
The rationale for these questions is the following: if, say, $\chi$ is primitive modulo a prime $q$ of order $d$ then $\{\chi(n)\}_{n < q}$ equidistributes among the $d$th order roots of unity (by orthogonality of Dirichlet characters). If $d$ is large then one expects that the level sets 
$$
\{n \leq x : \chi(n) = \alpha\}, \quad \alpha^d = 1, 
$$ 
should become sparse i.e., of size $o(x)$, even when $x$ is rather small relative to $q$. Naturally, we would like to understand how quickly this can occur (i.e. how small can $x$ be for this to happen).  The variation in the values of $\chi(n)$ also suggests the possibility that the short and maximal sums of $\chi$ might exhibit cancellation as well.

\subsection{Main Results}
Our main results address each of the three questions above. In the interest of clarity we defer relevant remarks about the theorems below to Section \ref{subsec:Rems}. \\
Our first main result addresses Question \ref{ques:fibres}, provided $q$ is prime and $d$ is squarefree. See Remarks \ref{rem:NonPrimeq} and \ref{rem:nonSfd} below regarding the necessity of these assumptions. \\
(Unless indicated otherwise, all implicit constants in this paper are absolute.) 
\begin{thmMain} \label{thm:levelSet}
Let $q$ be a large prime, let $d \geq 2$ be squarefree with $d|(q-1)$ and let $\chi$ be a primitive character modulo $q$ of order $d$. Then there is an absolute constant $c_1 > 0$ such that if
$$
\delta := \max\Big\{\left(\frac{\log\log (ed)}{c_1\log d}\right)^{1/2}, \, (\log q)^{-c_1}\Big\}
$$ 
and $x > q^{\delta}$ then for any\footnote{Here and below, we write $\mu_d$ to denote the set of $d$th order roots of unity.} $\alpha \in \mu_d$,
$$
|\{n \leq x : \chi(n) = \alpha\}| \ll \frac{x}{(\log\log d)^{1/16}}.
$$
\end{thmMain}

Our second main theorem is a mean-square estimate for short sums of order $d$ characters, showing that for \emph{typical} $ 1\leq \ell \leq d-1$ the partial sums of $\chi^\ell$ over intervals $[1,q^\delta]$ exhibit cancellation for any fixed $\delta >0$, as long as $d$ has no small prime factors. 
This shows that Question \ref{ques:Burg} has a positive answer for \emph{almost all} powers $\chi^\ell$ of $\chi$. See Remark \ref{rem:Harper} for a discussion of the strength of these bounds.
\begin{thmMain}\label{thm:meanSqShortSums}
Let $q$ be a large prime, let $d \geq 2$ with $d|(q-1)$ and let $\chi$ be a primitive character modulo $q$ of order $d$.
Define\footnote{Here and elsewhere, for $d \in \mb{N}$ we write $P^-(d)$ to denote the smallest prime factor of $d$, and $P^+(d)$ to denote the largest prime factor of $d$, with the conventions $P^-(1) = P^+(1) := 1$.}
$$
G = G(d) := \min\{P^-(d), \log\log (ed)\}.
$$
Then there is an absolute constant $c_2 > 0$ such that if
$$
\delta := \max\Big\{\left(\frac{\log\log (ed)}{c_2\log d}\right)^{1/2}, \, (\log q)^{-c_2}\Big\}
$$ 
and $x > q^{\delta}$ then
$$
\frac{1}{d}\sum_{1 \leq \ell \leq d-1} \left|\sum_{n \leq x} \chi^\ell(n)\right|^2 \ll x^2 \frac{(\log G)^2}{G^{2/15}}.
$$
\end{thmMain}
Our third main theorem
gives an upper bound for the average size of $M(\chi^\ell)$ with $1 \leq \ell \leq d-1$ that improves on the P\'{o}lya-Vinogradov inequality for \emph{any} $d |(q-1)$ having no small prime factors. 
This addresses Question \ref{ques:PV}, again for \emph{almost all} powers $\chi^\ell$ of $\chi$. 
\begin{thmMain} \label{thm:PVImpAvg}
Let $q$ be a large prime, let $d\geq 2$ with $d|(q-1)$ and let $\chi$ be a primitive character modulo $q$ of order $d$. Then
$$
\frac{1}{d}\sum_{1 \leq \ell \leq d-1} M(\chi^\ell) \ll (\sqrt{q}\log q)\left(\sqrt{\frac{\log\log\log q}{\log \log q}} + \frac{1}{P^-(d)}\right).
$$
\end{thmMain}
Combined with a (slight extension of a) result of Granville and Soundararajan on characters of \emph{fixed odd} order \cite[Thm. 2]{GSPret}, this gives the following bound, which provides a uniform saving for all \emph{prime} $d$, even when $d = d(q) \ra \infty$. See Remark \ref{rem:Oddness} for a discussion of the novelty of this result.
\begin{corMain}\label{cor:unifPVImp}
Let $q$ be a large prime. Then, uniformly over all primitive characters $\chi$ modulo $q$ of prime order $d \geq 3$ with $d|(q-1)$,
$$
\frac{1}{d}\sum_{1 \leq \ell \leq d-1} M(\chi^\ell) \ll (\sqrt{q}\log q) \left(\frac{\log\log\log q}{\log \log q}\right)^{1/2}.
$$
\end{corMain}

\subsection{More precise results}
Theorems \ref{thm:levelSet} and \ref{thm:meanSqShortSums} may be extended to a 
larger collection of completely multiplicative functions whose non-zero values are roots of unity of a large order. Considerations of such a general (but slightly different) nature arose also in \cite{GSUppL}.
\begin{def1}
Let $d \in \mb{N}$ and $x_0 \geq 1$. We say that a completely multiplicative function $f: \mb{N} \ra \mu_d\cup\{0\}$ \emph{weakly equidistributes beyond (a scale) $x_0$} if\footnote{The constant $100$ here could be replaced by any fixed constant, and is merely chosen for concreteness.}
for any $x \geq x_0$,
$$
\max_{\alpha \in \mu_d} |\{n \leq x : f(n) = \alpha\}| \leq \frac{100x}{d} \prod_{\ss{p \leq x \\ f(p) = 0}} \left(1-\frac{1}{p}\right).
$$
We define $\mc{M}(x_0;d)$ to be the collection of all such completely multiplicative functions. For $f \in \mc{M}(x_0;d)$ and $x \geq x_0$ we write
$$
c_f(x) := \prod_{\ss{p \leq x \\ f(p) = 0}}\left(1-\frac{1}{p}\right),
$$
and we call $x_0$ the \emph{threshold} of $f$. We define $\mc{M}_1(x_0;d)$ to be the subcollection of those $f \in \mc{M}(x_0;d)$ for which
\begin{equation}\label{eq:bddZeroSet}
\sum_{\ss{p: f(p) = 0}} \frac{1}{p} \leq 100,
\end{equation}
so that $c_f(x) \gg 1$ uniformly in $x$ whenever $f \in \mc{M}_1(x_0;d)$.
\end{def1}
\noindent For example, let $\chi$ be a primitive character modulo $q$ prime and order $d|(q-1)$. By \eqref{eq:Mchi}, we have
$$
\sum_{n \leq x} \chi^\ell(n) \ll \sqrt{q}\log q = o(x/d)
$$
with $x > q^{3/2+\e}$ \emph{uniformly} over $1 \leq \ell \leq d-1$ and $d|(q-1)$. Thus, by the Weyl criterion,
$$
\max_{\alpha \in \mu_d} |\{n \leq x : \chi(n) = \alpha\}| = (1+o(1)) \frac{x}{d}, \quad x > q^{3/2+\e}.
$$
Thus, $\chi \in \mc{M}_1(q^\theta;d)$ for any $\theta > 3/2$. \\ 
Our first general theorem shows that
 for any fixed $\delta > 0$ and any\footnote{In contrast to our main theorems, this result makes no assumptions on the arithmetic nature of $d$.} 
$d$ large enough relative to $\delta$, if $f \in \mc{M}_1(x_0;d)$ 
then $f(p) \neq 0,1$ for many primes $p \leq x_0^{\delta}$.
\begin{thm} \label{thm:HilApp}
Let $x_0 \geq 3$ be large and let $\eta,\delta \in (0,1)$. Then there are absolute constants $c_3 \in (0,1)$ and
$C_1,C_2,C_3 > 0$ such that the following is true.
\begin{enumerate}[(a)]
\item If $f \in \mc{M}_1(x_0;d)$ with\footnote{Here and elsewhere, $\rho: [0,\infty) \ra \mb{R}$ denotes the Dickman-de Bruijn function, defined uniquely by the initial condition $\rho(u) = 1$ for $u \in [0,1]$, and
$$
u\rho'(u) + \rho(u-1) = 0, \quad u > 1.
$$
See \cite[Sec. III.5.3-III.5.4]{Ten} for an account of some of its useful properties.}

\begin{equation}\label{eq:consD}
C_1 \rho(\tfrac{C_2}{\eta \delta})^{-1} \leq d \leq \tfrac{C_3}{\eta \delta} e^{(\log x_0)^{c_3}} 
\end{equation}
then for any $x > x_0^{\delta}$,
$$
\sum_{\ss{p \leq x \\ f(p) \neq 0,1}} \frac{1}{p} \geq \log(1/\eta).
$$
\item If $f$ is a primitive Dirichlet character modulo a large prime $q$ and $d|(q-1)$, and $\eta \delta \geq (\log q)^{-c}$ for some $0 < c  < c_3$, then the upper bound constraint in \eqref{eq:consD} may be removed.
\end{enumerate}
\end{thm}

Our second general theorem gives a mean-square estimate for short sums of powers $f^{\ell}$, for $f \in \mc{M}(x_0;d)$. It shows that if $f(p) \neq 0,1$ sufficiently often and $d$ has no small prime factors then the short sums of $f^\ell$ exhibit cancellation for \emph{most} $ 1\leq \ell \leq d-1$. 
\begin{thm}\label{thm:meanSqShortSumsParams}
Let $d \geq 2$ and $x_0\geq 3$.
Let $f \in \mc{M}(x_0;d)$, let $1 \leq x \leq x_0$ and set
$$
\Sigma := \min\Big\{P^-(d), 2 + \sum_{\ss{p \leq x \\ f(p) \neq 0,1}} p^{-1}\Big\}.
$$
Then 
$$
\frac{1}{d}\sum_{0 \leq \ell \leq d-1} \left|\sum_{n \leq x} f^\ell(n)\right|^2 \ll x^2 \frac{(\log \Sigma)^2}{\Sigma^{2/15}}.
$$
\end{thm}
In Section \ref{sec:ShortSums} we will combine Theorems \ref{thm:HilApp} and \ref{thm:meanSqShortSumsParams} to deduce the following corollary. 
\begin{cor} \label{cor:levelSet}
Let $q \geq 3$ be large. Let $\eta, \delta \in (0,1)$ and let $d \geq 2$ be squarefree. Then there are absolute constants $c_4 \in (0,1)$ and $C_4,C_5 > 0$ such that if
$\eta \delta > (\log q)^{-c_4}$ then the following holds: \\
If $\chi$ is a primitive character modulo a prime $q$ with order $d|(q-1)$ and
$$
d \geq C_4\rho(\tfrac{C_5}{\eta\delta})^{-1}
$$ 
then for any $x > q^{\delta}$ and $\alpha \in \mu_d$,
$$
|\{n \leq x : \chi(n) = \alpha\}| \ll \frac{x}{(\log(1/\eta))^{1/16}}.
$$
\end{cor}
Theorem \ref{thm:levelSet} immediately follows from Corollary \ref{cor:levelSet} and standard estimates for the Dickman-de Bruijn function $\rho$.

\subsection{Remarks on the results} \label{subsec:Rems}
\begin{rem}
Underlying the results above is the commonly-exploited strategy that while information about \emph{individual} characters is usually difficult to ascertain, it is often possible to make progress on average over a \emph{family} of characters. When $d$ is a large prime, for instance, the characters $\{\chi^\ell\}_{1 \leq \ell \leq d-1}$ all have exact degree $d$, and this collection, though thin, is large enough for averaging techniques to be effective. Fortunately, since these powers are all generated by $\chi$ we are able to use this average information to elucidate some properties of $\chi$, e.g., Theorem \ref{thm:levelSet}. 
\end{rem}

\begin{rem} \label{rem:NonPrimeq}
The condition that $q$ is prime in our main theorems is mainly a convenience that ensures that \eqref{eq:bddZeroSet} holds, and hence $\chi \in \mc{M}_1(x_0;d)$ for an appropriate scale $x_0$. As such, Theorems \ref{thm:HilApp} and \ref{thm:meanSqShortSumsParams} can be applied to $\chi$. However, the bound \eqref{eq:bddZeroSet} is only used in the proof of Theorem \ref{thm:HilApp}, and could be removed at the expense of replacing the quantity $\rho(\tfrac{C_2}{\eta\delta})$ by $\rho(\tfrac{q}{\phi(q)} \frac{C_2}{\eta\delta})$ in \eqref{eq:consD}. \\
As a result, we may equally well extend Corollary \ref{cor:levelSet} to a collection of moduli $q$ with uniformly bounded sums $\sum_{p|q}p^{-1}$.
Note in this connection that the trivial bound 
$$
\Big|\sum_{n \leq x} \chi(n)\Big| \leq |\{n \leq x : (n,q) = 1\}| \ll x\prod_{p|q}\left(1-\frac{1}{p}\right),
$$
valid for any $x > q^{\delta}$ and $\delta > 0$ fixed, shows that if $\sum_{p|q} p^{-1}$ is \emph{unbounded} as a function of $q$, in contrast, then we can trivially answer Question \ref{ques:Burg} in the affirmative.
\end{rem}
\begin{rem} \label{rem:nonSfd}
Our requirement that $d$ be squarefree in Theorem \ref{thm:levelSet} and Corollary \ref{cor:levelSet} is needed
to ensure that when $d$ is large so are most of its prime factors. As such, the group $\mu_d$ does not have ``too many'' small subgroups.
Morally, this prevents from occurring the situation that $\chi(n)$ has order much smaller than $d$ for many $n$, which would yield to much repetition in the sequence $(\chi(n))_n$. In place of the squarefreeness of $d$, it would be sufficient to assume that
$$
\prod_{\ss{p^k||d \\ p > z}}p^k, \text{ with } z \asymp \log(1/\eta),
$$
is large enough in terms of $\eta,\delta$.
\end{rem}

\begin{rem}\label{rem:Harper}
The savings obtained in Theorem \ref{thm:meanSqShortSums}, though non-trivial, are admittedly weak. By comparison, if we assume GRH then the far stronger \emph{square-root cancelling} bound
$$
\left|\sum_{n \leq x} \chi(n)\right| \ll_{\e} x^{1/2+\e}, \quad x > q^{\e}
$$ 
holds for individual character sums modulo $q$. Even unconditionally, if we average over \emph{all} characters $\chi \pmod{q}$ (i.e., the case $d = q-1$)
then far stronger results than Theorem \ref{thm:meanSqShortSums} can be proved. In particular, Harper \cite{Harp} has recently shown using the theory of random multiplicative functions that
$$
\frac{1}{q-1} \sum_{\ss{\chi \pmod{q}}} \left|\sum_{n \leq x} \chi(n)\right| \ll \frac{\sqrt{x}}{(\log\log \min\{x,q/x\})^{1/4}}.
$$
We might expect by analogy that
\begin{equation} \label{eq:l1davg}
\frac{1}{d}\sum_{0 \leq \ell \leq d-1} \left|\sum_{n \leq x} \chi^{\ell}(n)\right| = o_{d \ra \infty}(\sqrt{x}), 
\end{equation}
at least in some range of $x \in [1,q]$. It would be interesting to understand whether Harper's tools (suitably adapted to treat random multiplicative functions taking uniformly distributed values in $\mu_d$) could be used to study the average \eqref{eq:l1davg}, especially when $d$ grows only slowly with $q$.
\end{rem}

\begin{rem}\label{rem:Oddness}
Note that Theorem \ref{thm:PVImpAvg} is only non-trivial when $d$ has no small prime factors, and therefore odd. In view of \eqref{eq:PVoddImp}, Theorem \ref{thm:PVImpAvg} is thus much weaker than existing results when $d$ is slowly growing. However, note that the exponent 
$$
\delta(d) = 1-\frac{d}{\pi} \sin(\pi/d) \ll \frac{1}{d^2}, 
$$ 
so \eqref{eq:PVoddImp} is no stronger than \eqref{eq:Mchi} as soon as\footnote{Though the method of \cite{GSPret} on which \eqref{eq:PVoddImp} is based assumes $d$ as fixed, slight alterations of the argument yield to a result in which $d$ is allowed to grow with $q$; see Lemma \ref{lem:GSLargeG}.} $d \gg \sqrt{\log\log q}$. Theorem \ref{thm:PVImpAvg} and Corollary \ref{cor:unifPVImp} are therefore new in the range $d \gg \sqrt{\log\log q}$. \\
It is worth noting that Theorem \ref{thm:PVImpAvg} is also related to (though not implied by) \cite[Thm. 3]{GSUppL}, where, for slowly-growing $d$ an upper bound for the \emph{geometric mean} of the related quantities 
$$
|L(1,\chi^\ell)|, \quad 1 \leq \ell \leq d, \, (\ell,d) = 1
$$ 
is obtained that goes beyond the P\'{o}lya-Vinogradov bound. The estimate there does not, however, extend uniformly to the full range of $d$ considered here.
\end{rem}

\begin{rem}
We have made no attempt to optimize any of the exponents in Theorems \ref{thm:levelSet} or \ref{thm:meanSqShortSums}, and we do not believe our results to be best possible.
\end{rem}

\subsection{Plan of the paper}
The paper is structured as follows. \\
In Section \ref{sec:elem} we give context for our main theorems by providing elementary proofs of two results, Propositions \ref{prop:nonOne} and \ref{prop:largeArg}. These results show how assuming $d$ is large helps in finding small $n$ with $\chi(n)$ (at times significantly) different from 1 in value. \\
In Section \ref{sec:pret} we give a brief review of pretentious number theory, then in Section \ref{subsec:Deducts} we deduce Corollary \ref{cor:levelSet} and Theorem \ref{thm:levelSet} from the more general Theorems \ref{thm:HilApp} and \ref{thm:meanSqShortSumsParams}. As Theorem \ref{thm:meanSqShortSumsParams} is the more novel and involved of the two theorems,  we provide a sketch of the proof of that theorem in Section \ref{subsec:ProofStrat}. The proof itself appears in Section \ref{sec:Thm12proof}. \\
In Section \ref{sec:HilAppLow} we derive Theorem \ref{thm:HilApp}. Combining this with the work of Section \ref{sec:ShortSums},
we then deduce Theorem \ref{thm:meanSqShortSums}. \\
Finally, in Section \ref{sec:PVImp} we prove Theorem \ref{thm:PVImpAvg} and Corollary \ref{cor:unifPVImp} by combining (slightly extended) work of Granville and Soundararajan with some combinatorial observations related to \emph{sum-free} sets in abelian groups. \\
Sections \ref{sec:elem} and \ref{sec:PVImp} may be read independently of the remaining sections.

\subsection*{Acknowledgments}
Parts of this work were completed during visits by the author to the University of Bristol, to the Institut \'{E}lie Cartan de Lorraine and to King's College London. We would like to warmly thank Bristol, IECL and KCL for their hospitality and excellent working conditions. \\
We are most grateful to Andrew Granville, Oleksiy Klurman, Youness Lamzouri and Aled Walker for helpful discussions, references and encouragement. We also thank the anonymous referee for reading a previous version of this paper and for providing helpful comments. 

\section{Elementary Arguments Towards Small $n$ with $\chi(n)\neq 1$} \label{sec:elem}
Using only elementary arguments, in this section we will prove two results about large order characters that are only conjectural for characters of fixed order. This provides evidence that large order characters are easier to study than their fixed order counterparts, and motivates the investigations in the remainder of the paper. 
\subsection{Estimates for $n_{\chi}$}
Let $\delta \in (0,1)$. In this subsection we show that if $\chi$ is a primitive character modulo a prime $q$ with order $d$ sufficiently large in terms of $\delta$ then one can find solutions to $\chi(n) \neq 1$ with $n \leq q^{\delta}$. When $\delta < 1/4$ this goes beyond what can be obtained using Burgess' theorem. Such an observation has previously been made\footnote{Strictly speaking, Norton states his results as $n_{\chi} \ll q^{\frac{1}{4\alpha_w} + \e}$ for prime $q$ and $d \geq w$, where $\alpha = \alpha_w$ is the unique solution to $\rho(\alpha) = 1/w$. Aside from the factor $1/4$ that arises from his use of Burgess' theorem, it is easy to see that the parameter choices in Proposition \ref{prop:nonOne} correspond with his.} by Norton (see e.g., \cite[Thm. 6.4]{Nor} or \cite[Thm. 1.20]{Nor2}), but we give an alternative, short proof.
\begin{prop} \label{prop:nonOne}
Let $\delta > 0$ and let $q \geq q_0(\delta)$ be prime.  
If $\chi$ is a primitive character modulo $q$ of order $d > \rho(1/\delta)^{-1}$ then there is $1 \leq n \leq q^{\delta}$ with $\chi(n) \neq 1$. In particular, 
$$
n_{\chi} \ll q^{\frac{\log\log(Cd\log d)}{\log d}}
$$ 
for some $C > 0$ absolute.
\end{prop}
To prove this we use the following simple combinatorial lemma.
\begin{lem}
Let $\delta > 0$ and let $q \geq q_0(\delta)$ be prime. Then there is a constant $c = c(\delta) \in (0,1)$ for which the following holds. \\
If $0 \leq d_1 \leq d_2 \leq cq$ are such that 
$$
n^{d_2} \equiv n^{d_1} \pmod{q} \text{ for all } 1 \leq n \leq q^\delta
$$ 
then $d_1 = d_2$. In fact, we may take any $0 < c < \rho(1/\delta)$.
\end{lem}
\begin{proof}
Assume the contrary, so that $d := d_2-d_1 > 0$. By assumption, we have $t^d \equiv 1 \pmod{q}$ for all $1 \leq t \leq q^{\delta}$, and in particular for all primes $p \in [1,q^{\delta}]$ this congruence holds. But then since $(t_1t_2)^d \equiv 1 \pmod{q}$ whenever $t_j^d \equiv 1 \pmod{q}$ for $j = 1,2$, all $q^{\delta}$-friable\footnote{Given $y \geq 2$ we say that a positive integer $n$ is \emph{$y$-friable} if any prime factor $p|n$ must satisfy $p \leq y$.} integers $1 \leq t \leq q-1$ also satisfy this congruence. But for large enough $q$ there are $(\rho(1/\delta) + o(1)) q$ such integers up to $q$ \cite[Thm. III.5.8]{Ten}.
Thus, if $0 < c < \rho(1/\delta)$ and $q$ is large enough then we find that there are $> cq \geq d$ solutions to the polynomial equation $X^d \equiv 1 \pmod{q}$, which is a contradiction since $d \neq 0$. 
\end{proof}
\begin{proof}[Proof of Proposition \ref{prop:nonOne}]
Write $\chi = \chi_q^{\ell(q-1)/d}$,
where $\chi_q$ generates the character group modulo $q$ and $(\ell,d) = 1$. Setting $\chi_1 := \chi_q^{(q-1)/d}$, note that $\chi_1$ takes values in roots of unity of order $d$, and so if we can show that $\chi_1(n) \neq 1$ for some $n \leq q^{\delta}$ then the same is true for $\chi = \chi_1^\ell$. \\
Note that $1 \leq \frac{q-1}{d} \leq cq$ for some $0 < c < \rho(1/\delta)$. Now assume for the sake of contradiction that $\chi_1(n) = 1$ for all $1 \leq n \leq q^{\delta}$. Since $\chi_q$ is injective on $\mb{Z}/q\mb{Z}$ it follows that $n^{\tfrac{q-1}{d}} \equiv 1 \pmod{q}$ for all $1 \leq n \leq q^\delta$. By the previous lemma we deduce that $(q-1)/d = 0$, which is a contradiction. This establishes the first claim.\\
If $\delta$ is small enough then, using 
\begin{equation}\label{eq:DickmanLB}
\rho(u) \gg \left(\frac{e}{2u\log u}\right)^u 
\end{equation}
with $u = 1/\delta$ (see \cite[Sec. 3.9]{GraMSRI}) we deduce that if 
$$
\delta \log d > \log\log d + \log\log\log d + O(1)
$$ 
then $d > \rho(1/\delta)^{-1}$ and we may apply the first claim. The choice $\delta := \frac{\log\log (Cd\log d)}{\log d}$ with $C > 0$ absolute and sufficiently large, furnishes the second claim.
\end{proof}
\subsection{Small $n$ with $\chi(n)$ bounded away from $1$}
We next set out to study to what extent the values $\chi(n)$ can vary for $n \leq x$ and $x$ not much larger than $n_{\chi}$. It is possible that while $\chi(n) \neq 1$, $\chi(n)$ might still take values that are \emph{very close to} $1$ in this range, i.e., $\chi(n) = e(j/d)$ where\footnote{Given $t \in \mb{R}$ we write $e(t) := e^{2\pi i t}$ and $\|t\| := \min_{n \in \mb{Z}} |t-n|$.} $\|j/d\|$ is quite small. As a consequence, the partial sum of $\chi$ up to $x$ would not witness significant cancellation, contrary to expectations in line with the conjectured estimate \eqref{eq:chiVino}.  \\
By an elementary argument, however, we show that this is not necessarily the case when $d$ is large. In the sequel, for $z \in S^1$ we write $\text{arg}(z)$ to denote the element of $(-1/2,1/2]$ for which $z = e(\text{arg}(z))$.
\begin{prop} \label{prop:largeArg}
Let $q$ be a large prime, let $d \geq 2$ and let $\chi$ be a primitive character modulo $q$ of order $d$. Let $\delta \in (0,1/2)$ with 
$$
1/\delta =o\Big(\frac{\log\log q}{\log\log\log q}\Big)
$$ 
as $q \ra \infty$, and suppose that $d > \rho(1/\delta)^{-1}$. Then for any $c > 1$,
$$
\exists \, n \leq q^{\delta} \text{ with } |\text{arg}(\chi(n))| \geq \max\Big\{\tfrac{1}{d},\rho(1/\delta)^{c}\Big\}.
$$
\begin{rem}
For $c > 1$ fixed set $M := \rho(1/\delta)^{-c}$. Note that if $\chi(n) \neq 0,1$ then $|\text{arg}(\chi(n))| \geq 1/d$. Thus, when $M > d$ Proposition \ref{prop:largeArg} gives no further information than Proposition \ref{prop:nonOne} does.
The proposition is interesting, however, when $d$ is significantly large compared to $M$.
\end{rem}
\end{prop}
\noindent The proof requires a few auxiliary results. The first is due to Friedlander \cite[Thm. 1(B), Thm. 6(B)]{Fri}.
\begin{thm1}[Friedlander]
There is a continuous function $\sg$ defined on $\{(u,v) \in (1,\infty)^2 : v > u\}$ such that the following holds. For any fixed $\alpha > 1$ there is $u_0 = u_0(\alpha) > 2$ such that if $u \geq u_0$ and $X$ is large enough then
$$
|\{n \leq X : P^-(n) > X^{1/(\alpha u)}, \, P^+(n) \leq X^{1/u}\}| = \left(\sg(u,\alpha u) + O\left(\frac{1}{\log(X^{1/(\alpha u)})}\right)\right) \frac{X}{\log(X^{1/(\alpha u)})}.
$$
Moreover, under these conditions there is a $c_0 = c_0(\alpha) \geq 1$ such that $\sg$ satisfies 
\begin{equation}\label{eq:sgAlpha}
\sg(u,\alpha u) \geq  \rho(1/u) c_0^{-u}.
\end{equation}
\end{thm1}

\begin{cor} \label{cor:prodNum}
Let $q$ and $\delta$ be as in the statement of Proposition \ref{prop:largeArg}. Then there is an absolute constant $c_0 \geq 1$ such that 
$$
|\{n = mk < q : \, m \leq q^{\delta/10} \text{ and } p|k \Rightarrow q^{\delta/10} < p \leq q^{\delta}\}| \gg \rho(1/\delta) c_0^{-1/\delta} q.
$$
\end{cor}
\begin{proof}
Call $\mc{S}_{\delta}$ the set of $n = mk$ as above. Note that if $n \in \mc{S}_{\delta}$ then its representation $n = mk$ is uniquely determined. Now, for each $m \leq q^{\delta/10}$ define $u_m$ to be the unique solution to $(q/m)^{1/u_m} = q^{\delta}$; explicitly, $u_m = \delta^{-1}\Big(1-\frac{\log m}{\log q}\Big)$. Taking $\alpha = 10$ in Friedlander's theorem and using the lower bound \eqref{eq:sgAlpha}, we get
$$
|\mc{S}_\delta| \geq \sum_{m \leq q^{\delta/10}} \sum_{\ss{k < q/m \\ q^{\delta/10} < P^-(k) \leq P^+(k) \leq q^{\delta}}} 1 \gg \frac{q}{\delta \log q} \sum_{m \leq q^{\delta/10}} \frac{\sg(u_m, 10 u_m)}{m}
\geq \frac{q}{\delta \log q} \sum_{m \leq q^{\delta/10}} \frac{\rho(u_m)c_0^{-u_m}}{m},
$$
for large enough $q$. 
Since $c_0 \geq 1$, $\rho$ is a decreasing function, and $u_m \leq 1/\delta$ uniformly over $m \leq q^{\delta/10}$, 
we obtain
$$
|\mc{S}_\delta| \gg \frac{q\rho(1/\delta) c_0^{-1/\delta}}{\delta \log q} \sum_{m \leq q^{\delta/10}} \frac{1}{m} \gg q\rho(1/\delta)c_0^{-1/\delta},
$$
as claimed.
\end{proof}

\begin{lem}\label{lem:UnifDist}
Let $I \subseteq [0,1]$ be an open interval
with length $|I|$. If $\chi$ is a primitive character modulo prime $q$ of order $d$ then for any $K \geq 1$,
$$
|\{n < q : \text{arg}(\chi(n)) \in I\}| = q\left(|I| + O\left(\frac{1}{K} + \frac{\log(1+\llf K/d\rrf)}{d}\right)\right).
$$
\end{lem}
\begin{proof}
We apply the Erd\H{o}s-Tur\'{a}n inequality \cite[Thm. I.6.15]{Ten}. Given $K \geq 1$, 
$$
| \, |\{n < q : \text{arg}(\chi(n)) \in I\}| - q|I| \, | \ll \frac{q}{K} + \sum_{1 \leq k \leq K} \frac{1}{k}\left|\sum_{n \leq q} \chi(n)^k\right|.
$$
If $d|k$ then $\chi(n)^k$ is principal and the inner sum is $q-1$. Otherwise, if $d \nmid k$ then $\chi^k$ is non-principal and the sum is zero by orthogonality. This yields the upper bound
$$
\ll \frac{q}{K} + q\sum_{1 \leq k \leq K} \frac{1_{d|k}}{k} \ll q\left(\frac{1}{K} + \frac{\log(1+\llf K/d\rrf)}{d}\right),
$$
and implies the claim.
\end{proof}

\begin{proof}[Proof of Proposition \ref{prop:largeArg}]
Let $1 \leq M \leq d$ be a parameter to be chosen later. Assume for the sake of contradiction that $|\text{arg}(\chi(n))| \leq \frac{1}{M}$ for all $n \leq q^{\delta}$. It follows that whenever $n = mk < q$, where $m \leq q^{\delta/10}$ and $q^{\delta/10} < P^-(k) \leq P^+(k) \leq q^{\delta}$,
$$
|\text{arg}(\chi(mk))| \leq |\text{arg}(\chi(m))| + \sum_{p|k}|\text{arg}(\chi(p))| \leq \frac{1}{M} + \frac{10}{\delta} \max_{\ss{p|k \\ p \leq q^{\delta}}} |\text{arg}(\chi(q))| \leq \frac{11}{M\delta}.
$$
Since $M \leq d$, applying Lemma \ref{lem:UnifDist} with $K = \llf M\rrf -1 < d$ gives
$$
|\mc{S}_\delta| = |\{n = mk < q : m \leq q^{\delta/10}, \, p|k \Rightarrow q^{\delta/10} < p \leq q\}| \leq |\{n < q : |\text{arg}(\chi(n))| \leq \tfrac{11}{M\delta}\}| \ll \frac{q}{M\delta}.
$$
On the other hand, Corollary \ref{cor:prodNum} directly implies that 
$$
|\mc{S}_\delta| \gg q\rho(1/\delta) c_0^{-1/\delta}.
$$
If $\delta$ is sufficiently small and $C > 0$ is a large enough absolute constant then we obtain a contradiction with
$$
M := \min\{d,C (\delta\rho(1/\delta))^{-1} c_0^{1/\delta}\}.
$$ 
Thus, there must exist $n \leq q^\delta$ with $|\text{arg}(\chi(n))| > 1/M$. Since for any $c > 1$ and small enough $\delta$ the bound $\rho(1/\delta)^{c} \ll \delta \rho(1/\delta) c_0^{-1/\delta}$ holds, we get $1/M \gg \max\{1/d,\rho(1/\delta)^c\}$ and the claim follows.
\end{proof}
%

\section{Background and Proof Strategy} \label{sec:ShortSums}
\subsection{A Pretentious Primer} \label{sec:pret}
The arguments used towards the proof of our main theorems are grounded in notions of pretentious number theory, as developed by Granville and Soundararajan. Here, we give a brief overview of those ideas from that subject that will be relevant in this paper. \\
Let $\mb{U} := \{z \in \mb{C} : |z| \leq 1\}$. Given arithmetic functions $f,g: \mb{N} \ra \mb{U}$ and $x \geq 2$ we define the \emph{pretentious distance} between $f$ and $g$ (at scale $x$) as 
$$
\mb{D}(f,g;x) := \left(\sum_{p \leq x} \frac{1-\text{Re}(f(p)\bar{g}(p))}{p}\right)^{1/2}.
$$
Note that $\mb{D}(f,g;x) = \mb{D}(f\bar{g},1;x)$, and by Mertens' theorem, $0 \leq \mb{D}(f,g;x)^2 \leq 2\log\log x$. This distance function also satisfies a triangle inequality: given $f,g,h: \mb{N} \ra \mb{U}$ we have
\begin{equation}\label{eq:pretTri}
\mb{D}(f,h;x) \leq \mb{D}(f,g;x) + \mb{D}(g,h;x),
\end{equation}
which implies the useful inequality (see \cite[Lem. 3.1]{GSPret})
\begin{equation}\label{eq:pretTriProd}
\mb{D}(f_1f_2,g_1g_2;x) = \mb{D}(f_1\bar{g}_1, f_2\bar{g}_2;x) \leq \mb{D}(f_1\bar{g}_1,1;x) + \mb{D}(f_2\bar{g}_2,1;x) = \mb{D}(f_1,g_1;x) + \mb{D}(f_2,g_2;x).
\end{equation}
If $f$ and $g$ are multiplicative functions for which $\mb{D}(f,g;x)^2 = o(\log\log x)$ then $f(p) \approx g(p)$ for most $p$ (in a suitable average sense), and we think of $f$ and $g$ as approximating one another. In the particular case that $\mb{D}(f,g;x)$ is \emph{bounded} as a function of $x$ we say
that $f$ is \emph{$g$-pretentious} (or, symmetrically, that $g$ is \emph{$f$-pretentious}). \\
The pretentious distance can be used to express upper bounds for C\'{e}saro averages of bounded multiplicative functions. The Hal\'{a}sz-Montgomery-Tenenbaum inequality \cite[Cor. III.4.12]{Ten}, a quantitative refinement of fundamental work of Hal\'{a}sz, states that for a multiplicative function $f: \mb{N} \ra \mb{U}$ and parameters $x \geq 3$ and $T \geq 1$,
\begin{equation}\label{eq:HMT}
\frac{1}{x}\left|\sum_{n \leq x} f(n)\right| \ll (M+1) e^{-M} + \frac{1}{T} + \frac{\log\log x}{\log x},
\end{equation}
where $M := \min_{|t| \leq T} \mb{D}(f,n^{it};x)^2$. Thus, if $f$ is not $n^{it}$-pretentious for all $|t| \leq T$, and $T$ is large enough, then the partial sums of $f$ are small. This result will be used several times in the sequel. 

\subsection{Deductions of Corollary \ref{cor:levelSet} and Theorem \ref{thm:levelSet}} \label{subsec:Deducts} 
We show in this section that our main results on level sets, Theorem \ref{thm:levelSet} and Corollary \ref{cor:levelSet}, are consequences of our general Theorems \ref{thm:HilApp} and \ref{thm:meanSqShortSumsParams}. \\
Given a multiplicative function $f : \mb{N} \ra S^1 \cup \{0\}$ and $\alpha \in S^1$ define the level set
$$
\mc{A}_\alpha(x;f) := \{n \leq x : f(n) = \alpha\}, \quad x \geq 1.
$$
\begin{cor}\label{cor:fixedChinCes}
Assume the hypotheses and notation of Theorem \ref{thm:meanSqShortSumsParams}. Let $\{\alpha_j\}_{1 \leq j \leq d}$ be an ordering of $\mu_d$ so that 
\begin{align*}
|A_{\alpha_1}(x;f)| &= \max_{\alpha \in \mu_d} |\mc{A}_\alpha(x;f)|, \\
|\mc{A}_{\alpha_j}(x;f)| &= \max_{\ss{\alpha \in \mu_d \\ \alpha \notin \{\alpha_1,\ldots,\alpha_{j-1}\}}} |\mc{A}_{\alpha}(x;f)|, \quad j \geq 2. 
\end{align*}
Then for any $J \geq 1$, 
$$
|\mc{A}_{\alpha_J}(x;f)| \ll \frac{x}{\sqrt{J}}\frac{\log \Sigma}{\Sigma^{1/15}}.
$$
In particular,
$$
\max_{\alpha \in \mu_d} |\{n \leq x : f(n) = \alpha\}| \ll x \frac{\log \Sigma}{\Sigma^{1/15}}.
$$
\end{cor}
\begin{proof}[Proof of Corollary \ref{cor:fixedChinCes} assuming Theorem \ref{thm:meanSqShortSumsParams}]
The second claim follows from the first with $J = 1$ so it suffices to prove the first. \\
By orthogonality 
modulo $d$,
$$
\frac{1}{d}\sum_{0 \leq \ell \leq d-1} \left|\sum_{n \leq x} f^\ell(n)\right|^2 = \sum_{n,m \leq x} \frac{1}{d}\sum_{0 \leq \ell \leq d-1} (f(n)\bar{f}(m))^{\ell} = |\{n,m \leq x : f(n) = f(m) \neq 0\}|.
$$
Next, we note that by positivity,
\begin{align*}
|\mc{A}_{\alpha_J}(x;f)|^2 &\leq \frac{1}{J} \sum_{1 \leq j \leq J} |\mc{A}_{\alpha_j}(x;f)|^2 = \frac{1}{J}\sum_{\ss{n,m \leq x \\ f(n) = f(m) \\ f(n) \in \{\alpha_1,\ldots,\alpha_J\}}} 1 \\
&\leq \frac{1}{J} |\{n,m \leq x : f(n) = f(m) \neq 0\}|.
\end{align*}
Combining this with the previous equation and applying Theorem \ref{thm:meanSqShortSumsParams}, we obtain
$$
|\mc{A}_{\alpha_J} (x;f)| \ll \frac{1}{\sqrt{J}} \left(\frac{1}{d}\sum_{0 \leq \ell \leq d-1} \left|\sum_{n \leq x} f^\ell(n)\right|^2 \right)^{1/2} \ll \frac{x}{\sqrt{J}} \frac{\log \Sigma}{\Sigma^{1/15}},
$$
as claimed.
\end{proof}
\begin{proof}[Proof of Corollary \ref{cor:levelSet} assuming Theorems \ref{thm:HilApp} and \ref{thm:meanSqShortSumsParams}]
We may assume that $\eta \in (0,1)$ is smaller than any fixed constant, since otherwise the claim is trivial. Set now $z := \log(1/\eta)$ and factor $d = \mf{d} \mf{D}$, where
$$
\mf{d} := \prod_{\ss{p^k||d \\ p \leq z}} p^k, \quad \mf{D} := \prod_{\ss{p^k||d \\ p > z}} p^k.
$$
Define $\psi := \chi^\mf{d}$. Then $\psi$ has order $\mf{D}$, which satisfies $P^-(\mf{D}) > z$. 
Since $d$ is squarefree, by the prime number theorem we have $\mf{d} \leq e^{(1+o(1)) z} \leq \eta^{-2}$, and so by assumption,
$$
\mf{D} = d/\mf{d} \geq C_4 \eta^2 \rho\left(\frac{C_5}{\eta \delta}\right).
$$
Increasing $C_4,C_5$ if needed, we have $\mf{D} \geq C_1\rho\left(\frac{C_2}{\eta \delta'}\right)$, where $C_1,C_2$ are as in Theorem \ref{thm:HilApp}, and $\delta' := 2\delta/3$. \\
Now let $x > q^{\delta}$. Since $\psi \in \mc{M}_1(q^\theta;d)$ for $\theta > 3/2$ and prime $q$, taking any $0 < c_4 < c_3$ and applying Theorem \ref{thm:HilApp} gives
$$
\Sigma = \min\Big\{P^-(\mf{D}),2+\sum_{\ss{p \leq x \\ \psi(p) \neq 0,1}} p^{-1}\Big\} \geq \log(1/\eta).
$$
By Corollary \ref{cor:fixedChinCes}, if $\eta$ is sufficiently small then
$$
\max_{\alpha^{\mf{D}} = 1} |\mc{A}_{\alpha}(x;\psi)| \ll x\frac{\log \Sigma}{\Sigma^{1/15}} \ll \frac{x}{(\log(1/\eta))^{1/16}}.
$$
Since $\psi(n) = e(j/\mf{D})$ whenever $\chi(n) = e(j/d)$, we deduce that
$$
\max_{\beta^d = 1} |\mc{A}_{\beta}(x;\chi)| \leq \max_{\alpha^{\mf{D}} = 1} \sum_{\ss{\beta^d = 1: \\ \beta^\mf{d} = \alpha}} |\mc{A}_{\beta}(x;\chi)| = \max_{\alpha^{\mf{D}} = 1}  |\mc{A}_{\alpha}(x;\psi)| \ll \frac{x}{(\log(1/\eta))^{1/16}},
$$
as claimed.
%
%
\end{proof} 
\begin{proof}[Proof of Theorem \ref{thm:levelSet} assuming Theorems \ref{thm:HilApp} and \ref{thm:meanSqShortSumsParams}] 
Set $\eta = \delta$ and take $0 < c_1\leq c_4/2$ small, so that $\eta \delta \geq (\log q)^{-c_4}$ whenever $\delta \geq (\log q)^{-c_1}$. Using \eqref{eq:DickmanLB} and the hypothesis
$$
\frac{1}{\delta^2} \ll c_1^2 \frac{\log d}{\log\log(ed)},
$$
choosing $c_1$ smaller if needed we also have the required lower bound 
$$
d \geq C_4\rho(\tfrac{C_5}{\eta\delta})^{-1},
$$ 
with $C_4,C_5$ as in the statement of Corollary \ref{cor:levelSet}. 
Moreover, since $q > d$,
$$
\log(1/\eta) = \frac{1}{2}\min\Big\{c_1\log\log q, \log\log d - \log\log\log(ed)  + O(1) \Big\} \gg \log\log d.
$$
Thus, when $x > q^{\delta}$ Theorem \ref{thm:levelSet} follows from Corollary \ref{cor:levelSet}.
\end{proof} 

\subsection{Strategy of Proof of Theorem \ref{thm:meanSqShortSumsParams}} \label{subsec:ProofStrat}
Let $d \geq 2$, $x_0 \geq 3$ and $f \in \mc{M}(x_0;d)$. Let also $1 \leq x \leq x_0$. In Section \ref{sec:Thm12proof} we will prove Theorem \ref{thm:meanSqShortSumsParams}. Since its proof is somewhat involved, we will explain here our strategy towards its proof.
\subsubsection{Initial Setup} 
To prove Theorem \ref{thm:meanSqShortSumsParams} we will show that for a judicious choice of $\e = \e(d) > 0$,
\begin{equation}\label{eq:epsGoal}
\frac{1}{d}\sum_{0 \leq \ell \leq d-1} \left|\sum_{n \leq x} f^{\ell}(n)\right|^2 \ll \e^2 x^2.
\end{equation}
We will eventually show that we may take $\e \ll (\log \Sigma)/\Sigma^{1/15}$, from which the theorem follows. \\
In this direction, define
$$
\mc{C}_d(\e) := \Big\{1 \leq \ell \leq d-1 : \frac{1}{x} \Big|\sum_{n \leq x} f^\ell(n)\Big| \geq \e \Big\}.
$$
Since $|f| \leq 1$ it is immediately clear that
$$
\frac{1}{d}\sum_{0 \leq \ell \leq d-1} \left|\sum_{n \leq x} f^{\ell}(n)\right|^2 \leq \frac{1}{d} \sum_{\ss{1 \leq \ell \leq d-1 \\ \ell \notin \mc{C}_d(\e)}} \left|\sum_{n \leq x} f^{\ell}(n)\right|^2 + \frac{1}{d} \sum_{\ss{1 \leq \ell \leq d-1 \\ \ell \notin \mc{C}_d(\e)}} \left|\sum_{n \leq x} f^{\ell}(n)\right|^2 + \frac{x^2}{d} \leq (\e^2 + 1/d) x^2 + x^2\frac{|\mc{C}_d(\e)|}{d}. 
$$
We may suppose that $d > \e^{-2}$. Thus, if $|\mc{C}_d(\e)| \leq \e^2 d$ then \eqref{eq:epsGoal} is verified, and so our task is reduced to understanding the case $|\mc{C}_d(\e)| > \e^2 d$. \\
It turns out that this lower bound on $|\mc{C}_d(\e)|$ puts rigid constraints on $f$. Consequently, we show that for \emph{almost all} $\ell \in \mc{C}_d(\e)$ we still obtain \emph{some} cancellation in the partial sums of $f^{\ell}$, more precisely
$$
\frac{1}{x} \left|\sum_{n \leq x} f^{\ell}(n)\right| = o_{\e \ra 0}(x) \text{ for all but } o_{\e \ra 0}(d) \text{ values } \ell \in \mc{C}_{d}(\e).
$$
For this to be the case, according to \eqref{eq:HMT}
it would be sufficient to show that the minimal distances
\begin{equation}\label{eq:minDistsetup}
\mb{D}(f^\ell, n^{it_\ell};x) = \min_{|t| = O(\e^{-2})} \mb{D}(f^\ell, n^{it};x) 
\end{equation}
grow as a function of $1/\e$ for all but $o_{\e \ra 0}(d)$ powers $\ell$. 
We endeavour to verify this type of condition in the sequel.
\subsubsection{Proving the theorem assuming $t_{\ell} \equiv 0$} 
Our task turns out to be significantly simplified if we can show, roughly speaking, that $t_{\ell}$ may be replaced by $0$, or more precisely
\begin{equation}\label{eq:redtotlzero}
\mb{D}(f^{\ell},n^{it_{\ell}};x) = \mb{D}(f^{\ell},1;x) + O(1) \text{ for all $0 \leq \ell \leq d-1$.}
\end{equation}
Let us assume this is the case for the moment. Then we may bound the partial sums of $f^{\ell}$ in terms of the \emph{level sets}
$$
\sg_j = \sg_j(x) := \sum_{\ss{p \leq x \\ f(p) = e(j/d)}} \frac{1}{p}, \quad 1 \leq j \leq d-1,
$$
by decomposing
\begin{equation}\label{eq:Dfldecomp}
\mb{D}(f^\ell,1;x)^2 = \sum_{0 \leq j \leq d-1} \sum_{\ss{p \leq x \\ f(p) = e(j/d)}} \frac{1-\cos(2\pi j\ell/d) }{p} = \sum_{1 \leq j \leq d-1} (1-\cos(2\pi j\ell/d)) \sg_j.
\end{equation}
Note that while we know nothing about the sizes of the individual $\sg_j$, we do know that their \emph{sum} satisfies
$$
S_f(x) := \sum_{1 \leq j \leq d-1} \sg_j = \sum_{\ss{p \leq x \\ f(p) \neq 0,1}} \frac{1}{p} \geq \Sigma.
$$
We heuristically expect that the (non-zero) prime values $f(p)$ with $p \leq x$ are \emph{uniformly distributed} in $\mu_d$, so that each $\sg_j$ should be of roughly the same size $\sg_j \approx S_f(x)/d$. In particular, $\sg_j$ should be \emph{small} relative to $S_f(x)$ as $d \ra \infty$ for every $j$. However, we do not know that this is the case in practice. As is reflected in the bound in Theorem \ref{thm:meanSqShortSumsParams}, we instead seek lower bounds for $\mb{D}(f^\ell,1;x)^2$ in terms of $S_f(x)$, at least for \emph{most} $0 \leq \ell \leq d-1$. \\
Using a simple Fourier analytic argument we are able to show (Lemma \ref{lem:SmallSigCase}) that if the prime factors of $d$ are all large in terms of $\e$, and \emph{each} $\sg_j$ satisfies $\sg_j < \e S_f(x)$ then for most $1 \leq \ell \leq d-1$,
$$
\sum_{1 \leq j \leq d-1} \|\ell j/d\|^2 \sg_j = \left(\int_0^1 \|t\|^2 dt + o_{\e \ra 0}(1) \right) S_f(x) = (1/12 + o_{\e \ra 0}(1))S_f(x).
$$
Thus using $1-\cos (2\pi x) \geq 8 \|x\|^2$ in \eqref{eq:Dfldecomp}, for most $\ell$ we obtain
$$
\mb{D}(f^{\ell},1;x)^2 \geq (2/3 + o_{\e \ra 0}(1)) S_f(x).
$$
As $S_f(x) \geq \Sigma$, \eqref{eq:HMT} then yields an estimate of the shape
$$
\frac{1}{x}\left|\sum_{n \leq x} f^{\ell}(n)\right| \ll \e^2 + \Sigma e^{-c \Sigma},
$$
for some $c > 0$. This bound is more than sufficient. \\
The possibility remains that \emph{some} $\sg_{j_0}$ is large in the sense that $\sg_{j_0} \geq \e S_f(x)$. We show in this case (see Proposition \ref{prop:TKAppCes}) that if $|\mc{C}_d(\e)| > \e^2 d$ and $\ell \in \mc{C}_d(\e)$ then as long as $\ell j_0/d \pmod{1}$ is $\gg_{\e} 1$, thus bounded away from zero, we still obtain
\begin{equation}\label{eq:smallfellSum}
\frac{1}{x}\left|\sum_{n \leq x} f^{\ell}(n)\right| = o_{\e \ra 0}(1).
\end{equation}
To prove this we use the Tur\'{a}n-Kubilius inequality \cite[Thm. III.3.1]{Ten} and the complete multiplicativity of $f$ to obtain a decomposition
$$
\frac{1}{x}\sum_{n \leq x} f^{\ell}(n) \approx \frac{1}{\sg_{j_0}x}\sum_{\ss{mp \leq x \\ f(p) = e(j_0/d)}} f^{\ell}(pm) = \frac{e(\ell j_0/d)}{\sg_{j_0}} \sum_{\ss{p \leq x \\ f(p) = e(j_0/d)}} \frac{1}{p} \cdot \frac{p}{x}\sum_{m \leq x/p} f^\ell(m).
$$
Since the normalized partial sums $y \mapsto y^{-1} \sum_{n \leq y} g(n)$ of a multiplicative function $g$ are known to be slowly-varying\footnote{This is an oversimplification; in order to apply the appropriate Lipschitz estimates we must first twist $f^{\ell}$ by a suitable character $n^{iy_{\ell}}$; luckily, we may show that $|y_{\ell}|$, like $|t_{\ell}|$, is small and therefore negligible in the arguments.}  with $y$, we show roughly speaking that
$$
\frac{p}{x}\sum_{m \leq x/p} f^{\ell}(m) \approx \frac{1}{x}\sum_{m \leq x} f^{\ell}(m)
$$
uniformly in the range $p \leq x^{o(1)}$. Combined with the decomposition above, this leads to an estimate of the shape
$$
\frac{1}{x}\sum_{n \leq x} f^{\ell}(n) \approx e(\ell j_0/d) \frac{1}{x}\sum_{n \leq x} f^{\ell}(n).
$$
It is not hard to show (e.g. using the Erd\H{o}s-Tur\'{a}n inequality) that $|e(\ell j_0/d)-1| \gg_{\e} 1$ for all but $o_{\e \ra 0}(d)$ values $1 \leq \ell \leq d-1$, which then forces \eqref{eq:smallfellSum} to hold, as required.
\subsubsection{Reducing to the case $t_{\ell} \equiv 0$}
It remains to show that \eqref{eq:redtotlzero} holds, and so effectively $t_{\ell} = 0$ for all $\ell$. What we actually prove (see Proposition \ref{prop:approxHom}) is that $|t_{\ell}| \ll_{\e} \frac{1}{\log x}$ uniformly in $\ell$. \\
%
To motivate this, suppose for convenience that $1 \in \mc{C}_d(\e)$. Then $f$ is $n^{it_1}$-pretentious for some $|t_1| \ll \e^{-2}$. Now assume instead that $|t_1|$ were bounded away from 0. Then $f(p) \approx p^{it_1}$ for \emph{typical} primes $p$, and (at least for $\ell$ not too large, see Remark \ref{rem:triIneqApp}), $f^{\ell}$ should be $n^{i\ell t_1}$ pretentious and $t_{\ell} \approx \ell t_1$.
Now if $\ell \in \mc{C}_d(\e)$ then $f^\ell$ must have large partial sums as well. On the other hand, it can be shown (see \eqref{eq:twistChiEll} below) that
$$
\left|\sum_{n \leq x} f^\ell(n)\right| \ll \frac{x}{1+|t_{\ell}|} \ll \frac{x}{1+\ell |t_1|}.
$$
Thus, $\ell|t_1|$ cannot be large.
However, $|\mc{C}_d(\e)|$ contains many large values of $\ell$, thus $|t_1|$ itself must be quite small. \\
Unfortunately this argument is too simplistic, as when $\ell$ is large the powers $(f(p)p^{-it_1})^{\ell}$ may be significantly different from $1$ even if the values $f(p)p^{-it_1}$ typically are not. To make it rigorous we appeal to the theory of sumset arithmetic in additive combinatorics. Using an inverse sumset result due to Freiman \cite{Fre}, we show that if $d$ has no small prime factors then \emph{every} $0 \leq \ell \leq d-1$ has an \emph{efficient representation}
$$
\ell \equiv \ell_1 + \cdots + \ell_m,
$$
where $\ell_j \in \mc{C}_d(\e)$ for each $1 \leq j \leq m$ and $m = O_{\e}(1)$ (see Corollary \ref{cor:CDGen}). Under these conditions, we leverage properties of the pretentious distance to show that
$$
t_{\ell} = t_{\ell_1} + \cdots + t_{\ell_m} + O_{\e}\left(\frac{1}{\log x}\right),
$$ 
and as a result, that the map $\phi: \mb{Z}/d\mb{Z} \ra \mb{R}$ given by $\phi(\ell) := t_{\ell}$ satisfies the \emph{approximate homomorphism} condition
$$
|t_{\ell_1+\ell_2} - t_{\ell_1} - t_{\ell_2}| \ll_{\e} \frac{1}{\log x}.
$$
By applying a result due to Ruzsa on approximate homomorphisms \cite{Ruz}, we find that there is a \emph{genuine} homomorphism $\psi: \mb{Z}/d\mb{Z} \ra \mb{R}$ such that 
$$
\max_{\ell \in \mb{Z}/d\mb{Z}} |t_{\ell}-\psi(\ell)| \ll_{\e} \frac{1}{\log x}.
$$
Since $\mb{Z}/d\mb{Z}$ is a finite group and $\mb{R}$ is torsion-free, $\psi$ must be identically zero, which leads to $\max_{\ell \in \mb{Z}/d\mb{Z}} |t_{\ell}| \ll_{\e} 1/\log x$, as claimed.

\section{Proof of Theorem \ref{thm:meanSqShortSumsParams}} \label{sec:Thm12proof}
Following the outline provided in Section \ref{subsec:ProofStrat}, we prove Theorem \ref{thm:meanSqShortSumsParams} in this section.
\subsection{The structure of the minimizers $t_{\ell}$} \label{subsec:tlZero}
In this subsection we show the following proposition, which bounds the minimizers $t_{\ell}$ uniformly over $\ell$ under the assumptions that $|\mc{C}_d(\e)|$ is large and $d$ has no small prime factors.
\begin{prop}\label{prop:approxHom}
Let $d$ be a positive integer and let $c > 0$ be chosen such that $P^-(d) > 1/c$.
Set $m := \lceil 2/c^2 \rceil$ and for each $1 \leq \ell \leq d-1$ choose $t_\ell = t_{\ell}(\e) \in [-2m/\e^2,2m/\e^2]$ such that
$$
\mb{D}(f^\ell,n^{it_\ell};x) = \min_{|t| \leq (2m)/\e^2} \mb{D}(f^\ell,n^{it};x).
$$ 
If $\e$ is sufficiently small, $|\mc{C}_d(\e)| \geq c d$ and $m^2 \log(1/\e) < \tfrac{1}{32}\log\log x$ then
$$
\max_{1 \leq \ell \leq d-1} |t_\ell| \leq \frac{3\e^{-32m^2}}{\log x}.
$$
\end{prop}
\begin{rem} \label{rem:triIneqApp}
If $d$ grows sufficiently slowly then a simpler argument would suffice. 
By the minimal property of $t_\ell$ and repeated applications of \eqref{eq:pretTriProd},
$$
\mb{D}(1,n^{i(t_\ell-\ell t_1)};x) \leq \mb{D}(f^\ell,n^{it_\ell};x) + \mb{D}(f^\ell,n^{i\ell t_1};x) \leq 2\mb{D}(f^\ell,n^{i\ell t_1};x) \leq 2\ell \mb{D}(f,n^{it_1};x).
$$
If $d = o(\sqrt{\log\log x})$ and $1 \in \mc{C}_d(\e)$ then (using \eqref{eq:HMT}) the right-hand side is $o(\sqrt{\log\log x})$. It can then be shown that 
$$
t_\ell = \ell t_1 + O_{\e}(1/\log x) \text{ for all } 1 \leq \ell \leq d+1. 
$$
Since $t_1 = t_{d+1}$, we deduce that $|t_1| = O_{\e}(1/\log x)$ as a result. \\
The novelty of Proposition \ref{prop:approxHom} is that the same conclusion still holds, even when $d$ is fairly large, provided many of the powers $f^\ell$ have large partial sums.
\end{rem}

To prove Proposition \ref{prop:approxHom} we will need the following inverse sumset result, which follows from classical work of Freiman \cite{Fre} in additive combinatorics (see \cite{Wal} for an accessible proof).
\begin{lem} \label{lem:FreiApp}
Let $c > 0$. Let $G$ be a finite Abelian group and let $A\subset G$ be a symmetric subset\footnote{We say that a subset $S$ of a finite Abelian group is \emph{symmetric} whenever $s \in S$ if, and only if, $-s \in S$. Equivalently, $-S := \{-s : s \in S\} = S$.} that satisfies $|A| \geq c|G|$. If $A$ is not contained in a coset of a proper subgroup of $G$ then\footnote{Given a finite abelian group $G$, sets $A,B \subseteq G$ and an integer $r \geq 2$ we define the sumset $A+B$ as
$$
A+B := \{a+b : \, (a,b) \in A\times B\},
$$
and inductively define $rA := A + (r-1)A$.}
 $2^k A = G$ for some $1 \leq k \leq \lceil \frac{\log(1/c)}{\log(3/2)}\rceil$.
\end{lem}
\begin{proof}
Call $K := 1+\lceil \frac{\log(1/c)}{\log(3/2)}\rceil$. Since $A$ is symmetric, we have $A = -A$. Now, by Freiman's theorem, we see that if $B\subset G$ is symmetric and $B$ is not contained in a coset of a proper subgroup of $G$ then either $B+B = G$ or else $|B+B| \geq \frac{3}{2}|B|$. Applying this iteratively, we find that if $j\geq 1$ and we assume that none of the sets $A,2A,\ldots,2^{j-1}A$ is contained in a coset of a proper subgroup of $G$ then either $2^j A = G$, or else
$$
|G| \geq |2^jA| \geq \frac{3}{2}|2^{j-1}A| \geq \cdots \geq \left(\frac{3}{2}\right)^j |A| \geq \left(\frac{3}{2}\right)^j c|G|.
$$
If $j \geq K$ this is impossible, and so we deduce that $2^j A = G$ for some $1 \leq j \leq K-1$, as long as we can verify that $2^j A$ is not contained in a coset of a proper subgroup of $G$ for any $1 \leq j \leq K-1$. \\
But observe that if $B+B \subseteq b+H$ for some $H < G$ and $b \in G$ then for any $b' \in B$ we have $B \subseteq (b-b') + H$. It follows by induction that if $2^jA$ were contained in a coset of a proper subgroup of $G$ for some $j \geq 1$ then the same is true of $2^i A$ for any $0 \leq i \leq j$. Since, by hypothesis, $A$ is not contained in a coset of a proper subgroup of $G$ we may conclude that none of the iterated sets $2^j A$ are, and the conclusion follows.
\end{proof}

\begin{cor} \label{cor:CDGen}
Let $c > 0$ and let $d \geq 1$ be an integer such that $P^-(d) > 1/c$. If $A \subseteq \mb{Z}/d\mb{Z}$ is symmetric with $|A| \geq cd$ then $2^j A = \mb{Z}/d\mb{Z}$ for some $1 \leq j \leq \lceil \tfrac{\log(1/c)}{\log(3/2)}\rceil$.
\end{cor}
\begin{proof}
If $A = \mb{Z}/d\mb{Z}$ then the result is trivial, so we may assume instead that $A$ is a proper subset of $\mb{Z}/d\mb{Z}$. By Lemma \ref{lem:FreiApp}, it suffices to verify that any symmetric subset $A$ in $\mb{Z}/d\mb{Z}$ cannot be a subset of a coset of a proper subgroup of $\mb{Z}/d\mb{Z}$. But if this were the case then $A \subset b + H$ for some $H < G$ and $b \in G$, in which case
$$
c d \leq |A| \leq |b+H| = |H|,
$$
whereas if $H \neq G$ then $|H| = d'$ for some $d'|d$, $d' < d$. It follows that $d' \leq d/P^-(d) < cd$, which is a contradiction. The claim follows.
\end{proof}

\begin{proof}[Proof of Proposition \ref{prop:approxHom}]
Write $\mc{C}_d = \mc{C}_d(\e)$. By \eqref{eq:HMT}, for each $\ell \in \mc{C}_d$ we have
$$
\e x \leq \left|\sum_{n \leq x} f^\ell(n)\right| \ll x\left(\mb{D}(f^\ell, n^{i\tilde{t}_\ell};x)^2 e^{-\mb{D}(f^\ell,n^{i\tilde{t}_\ell};x)^2} + \e^2 + \frac{1}{\log x}\right),
$$
where here $\tilde{t}_{\ell} \in [-1/\e^2,1/\e^2]$ is chosen to minimize $\min_{|t| \leq 1/\e^2}\mb{D}(f^\ell,n^{it};x)$ (note that $t_\ell \neq \tilde{t}_\ell$ in general). If $\e$ is small enough (and thus $x$ is large enough), then upon rearranging we deduce that 
\begin{equation} \label{eq:conseqHMT}
\mb{D}(f^\ell,n^{i\tilde{t}_\ell};x)^2 \leq 2\log(1/\e)  \text{ for all } \ell \in \mc{C}_d.
\end{equation}
Now, since $P^-(d) > 1/c$ and $|\mc{C}_d| \geq cd$, by Corollary \ref{cor:CDGen} we know that $\mb{Z}/d\mb{Z}  = J\mc{C}_d$ for some $J \leq 2^{\lceil \frac{\log(1/c)}{\log(3/2)}\rceil} \leq 2c^{-2} \leq m$. Thus, for any $1 \leq \ell \leq d-1$ we have a representation
$$
\ell \equiv r_1 + \cdots + r_{\lambda} \pmod{d},
$$
where $r_j \in \mc{C}_d$ and $1 \leq \lambda \leq m$. But by \eqref{eq:conseqHMT}, \eqref{eq:pretTriProd} and induction,
$$
\mb{D}(f^{r_1 + \cdots + r_\lambda}, n^{i(\tilde{t}_{r_1}+\cdots + \tilde{t}_{r_\lambda})};x) \leq \sum_{1 \leq j \leq \lambda} \mb{D}(f^{r_j},n^{i\tilde{t}_{r_j}};x) \leq m \sqrt{2\log(1/\e)}.
$$ 
Since $\left|\sum_{1 \leq j \leq \lambda} \tilde{t}_{r_j}\right| \leq 2m/\e^2$, by the minimality property for $t_\ell \in [-2m/\e^2,2m/\e^2]$ \eqref{eq:pretTri} yields
$$
\mb{D}(n^{it_{\ell}},n^{i(\tilde{t}_{r_1} + \cdots + \tilde{t}_{r_\lambda})};x) \leq \mb{D}(f^\ell,n^{it_{\ell}};x) + \mb{D}(f^{r_1+\cdots + r_\lambda},n^{i(\tilde{t}_{r_1} + \cdots + \tilde{t}_{r_\lambda})};x) \leq 2m\sqrt{2\log(1/\e)}.
$$
By the Vinogradov-Korobov zero-free region for the Riemann zeta function, it follows that if $|t| \geq 100$ then for large $x$,
\begin{equation}\label{eq:big1nit}
\mb{D}(1,n^{it};x)^2 = \sum_{p \leq x} \frac{1-\text{Re}(p^{it})}{p} = \log\log x - \log|\zeta(1+1/\log x + it)| + O(1) \geq \frac{1}{4} \log\log x,
\end{equation}
so with $u_\ell := t_\ell - \tilde{t}_{r_1}-\cdots - \tilde{t}_{r_\lambda}$ this would give the contradiction 
\begin{equation*}
\frac{1}{4}\log\log x \leq \mb{D}(1,n^{iu_\ell};x)^2 \leq (2m\sqrt{2\log(1/\e)})^2 = 8m^2\log(1/\e),
\end{equation*}
if $|u_\ell| \geq 100$. Thus, we may assume that $|u_\ell| \leq 100$. 
Taking squares and instead using 
\begin{equation}\label{eq:small1nit}
\mb{D}(1,n^{it};x)^2 = \log(1+|t| \log x) + O(1)
\end{equation}
whenever $|t| \leq 100$, say, we thus deduce that
$$
\Big|t_{\ell} - \sum_{1 \leq j \leq \lambda} \tilde{t}_{r_j} \Big| \leq \frac{\e^{-8m^2}}{\log x}.
$$
Now let $1 \leq \ell_1,\ell_2 \leq d-1$. Choose any additive representations
$$
\ell_1 \equiv r_1+\cdots + r_{\lambda_1} \pmod{d}, \quad \ell_2 \equiv s_1 + \cdots + s_{\lambda_2} \pmod{d},
$$
with $\lambda_1,\lambda_2 \leq m$, then by the same argument applied with $\ell \in \{\ell_1,\ell_2, \ell_1+\ell_2\}$ we obtain, uniformly over $\ell_1,\ell_2 \in \mb{Z}/d\mb{Z}$,
\begin{align*}
&|t_{\ell_1 + \ell_2} - t_{\ell_1} - t_{\ell_2}| \\
&\leq |t_{\ell_1+\ell_2} - (\tilde{t}_{r_1} + \cdots + \tilde{t}_{r_{\lambda_1}} + \tilde{t}_{s_1} + \cdots + \tilde{t}_{s_{\lambda_2}})| + |t_{\ell_1} - (\tilde{t}_{r_1} + \cdots + \tilde{t}_{r_{\lambda_1}})| + |t_{\ell_2} - (\tilde{t}_{s_1} + \cdots + \tilde{t}_{s_{\lambda_2}})| \\
&\leq \frac{\e^{-8(2m)^2}}{\log x} + 2 \frac{\e^{-8m^2}}{\log x} \leq 3\frac{\e^{-32m^2}}{\log x}.
\end{align*}
Since we may always select $t_0 := 0$, and as $f^{\ell+kd} = f^\ell$ we can choose $t_{\ell+kd} = t_\ell$ for all $k \in \mb{Z}$, the map $\phi : \mb{Z}/d\mb{Z} \ra \mb{R}$ given by $\phi(m) := t_m$ satisfies the approximate homomorphism condition
$$
|\phi(\ell_1+\ell_2) - \phi(\ell_1) - \phi(\ell_2)| \leq 3 \frac{\e^{-32m^2}}{\log x}, \quad \ell_1,\ell_2 \in \mb{Z}/d\mb{Z}.
$$
By a generalization due to Ruzsa of a result of Hyers \cite[Statement (7.3)]{Ruz}, there is a genuine homomorphism $\psi: \mb{Z}/d\mb{Z} \ra \mb{R}$ such that
$$
\max_{\ell \in \mb{Z}/d\mb{Z}} |\phi(\ell) - \psi(\ell)| \leq 3 \frac{e^{-32m^2}}{\log x}.
$$
But there are no non-trivial homomorphisms from $\mb{Z}/d\mb{Z}$ to $\mb{R}$ since the latter is torsion-free. Hence, $\psi(\ell) = 0$ for all $\ell$ and we deduce that 
$$
\max_{1\leq \ell \leq d-1} |t_{\ell}| \leq \frac{3\e^{-32m^2}}{\log x},
$$
as claimed.
\end{proof}

\subsection{Studying the distances $\mb{D}(f^{\ell},1;x)$ using the level sets of $f(p)$} \label{subsec:tell1arg}
Having shown that the minimizers $|t_{\ell}|$ are uniformly small, we next study the sizes of the distances $\mb{D}(f^{\ell},1;x)$, which control the partial sums of $f^{\ell}$. 
We write
$$
\sg_j = \sg_j(x) := \sum_{\ss{p \leq x \\ f(p) = e(j/d)}} \frac{1}{p}, \quad 1 \leq j \leq d-1.
$$
Our analysis now splits into two cases, according to how large each $\sg_j$ is relative to the sum
$$
S_f(x) := \sum_{1 \leq j \leq d-1} \sg_j.
$$
\subsubsection{Case 1: each $\sg_j$ is small}
When each $\sg_j$ is small relative to $S_f(x)$ the following lemma provides lower bounds on the distances $\mb{D}(f^\ell,1;x)^2$, for almost all $1 \leq \ell \leq d-1$. 
\begin{lem}\label{lem:SmallSigCase}
Let $\e \in (0,1)$ be small and satisfy $\e P^-(d) > 1$, and assume that 
$$
\max_{1 \leq r \leq d-1} \sg_r < \e S_f(x).
$$ 
Then for all but $O(\e^{1/2}\log^2(1/\e) d)$ choices of $1 \leq \ell \leq d-1$ we have
$$
\mb{D}(f^{\ell},1;x)^2 \geq \left(\frac{2}{3} + O(\log(1/\e)^{-1})\right) S_f(x).
$$
\end{lem}
\begin{proof}
Given $t \in \mb{R}$, observe first of all the inequality
$$
1-\cos(2\pi t) = 2\sin^2(\pi \|t\|) \geq 8 \|t\|^2.
$$
It follows that for any $\ell \neq 0$,
\begin{equation}\label{eq:distToFrac}
\mb{D}(f^\ell,1;x)^2 \geq \sum_{\ss{p \leq x \\ f(p) \neq 0,1}} \frac{1-\text{Re}(f^{\ell}(p))}{p} = \sum_{1 \leq j \leq d-1} (1-\cos(2\pi j\ell /d)) \sum_{\ss{p \leq x \\ f(p) = e(j/d)}} \frac{1}{p} \geq 8 \sum_{1 \leq j \leq d-1} \Big\|\frac{\ell j }{d} \Big\|^2 \sg_j.
\end{equation}
We now seek upper bounds for the variance
$$
\Delta := \frac{1}{d} \sum_{1 \leq \ell \leq d} \left(\sum_{1 \leq j \leq d-1} \Big\|\frac{\ell j}{d} \Big\|^2 \sg_j - S_f(x) \int_0^1 \|t\|^2 dt  \right)^2 = \frac{1}{d} \sum_{1 \leq \ell \leq d} \left(\sum_{1 \leq j \leq d-1} \Big\|\frac{\ell j}{d} \Big\|^2 \sg_j - \frac{1}{12}S_f(x) \right)^2.
$$
The 1-periodic function $t\mapsto \|t\|^2$ has the absolutely convergent Fourier series
$$
\|t\|^2 = \frac{1}{12} + \frac{1}{2\pi^2} \sum_{r \neq 0} \frac{(-1)^r}{r^2} e(rt).
$$
Thus, expanding the square in $\Delta$ and inputting this expression yields
\begin{align*}
\Delta &= \sum_{1 \leq j_1,j_2 \leq d-1} \sg_{j_1}\sg_{j_2} \frac{1}{d} \sum_{1 \leq \ell \leq d} \left(\Big\|\frac{\ell j_1}{d}\Big\|^2 - \frac{1}{12}\right)\left(\Big\|\frac{\ell j_2}{d}\Big\|^2 - \frac{1}{12} \right) \\
&= \frac{1}{4\pi^4}\sum_{r_1,r_2 \neq 0} \frac{(-1)^{r_1+r_2}}{(r_1r_2)^2} \sum_{1 \leq j_1, j_2 \leq d-1} \sg_{j_1}\sg_{j_2} \frac{1}{d}\sum_{1 \leq \ell \leq d} e\left(\frac{\ell}{d}(j_1r_1-j_2r_2)\right) \\
&= \frac{1}{4\pi^4}\sum_{1 \leq j_1, j_2 \leq d-1} \sg_{j_1}\sg_{j_2} \sum_{\ss{r_1,r_2 \neq 0 \\ j_1r_1 \equiv j_2 r_2 \pmod{d}}} \frac{(-1)^{r_1+r_2}}{(r_1r_2)^2} .
\end{align*}
Note first of all that in the inner double sum, if $m = \max\{(r_1,d), (r_2,d)\} > 1$ then $m \geq P^-(d)$, and thus also $\max\{|r_1|,|r_2|\} \geq P^-(d)$. Since $(r_1r_2,d) > 1$ if and only if $\max\{(r_1,d),(r_2,d)\} > 1$, the contribution to $\Delta$ from $r_1,r_2 \neq 0$ with $(r_1r_2,d) > 1$ is thus
$$
\ll \sum_{1 \leq j_1,j_2 \leq d-1} \sg_{j_1}\sg_{j_2} \sum_{\ss{r_1,r_2 \neq 0 \\ \max\{|r_1|,|r_2|\} \geq P^-(d)}} \frac{1}{(r_1r_2)^2} \ll \frac{1}{P^-(d)} S_f(x)^2.
$$
We next consider the contribution from those $r_1,r_2 \neq 0$ with $(r_1r_2,d) = 1$. 
Fix $1 \leq j_1,j_2 \leq d-1$ for the moment, and suppose $j_1r_1 \equiv j_2r_2 \pmod{d}$. Then $(j_1,d) = (j_2,d) =: \lambda$. Since $\lambda < d$ and $\lambda|d$ we have $d/\lambda \geq P^-(d)$.  Setting $J_i := j_i/(j_1,j_2)$ for $i = 1,2$, the congruence is equivalent to $J_1r_1 \equiv J_2r_2 \pmod{d/\lambda}$, with $(J_1,J_2) = (J_1J_2,d/\lambda) = 1$. 
%
Separating the contribution of $\max\{|r_1|,|r_2|\} \geq d/\lambda$ from the remainder, we find
\begin{align}
\Delta &= \frac{1}{4\pi^4} \sum_{\ss{\lambda | d \\ \lambda < d}} \sum_{\ss{1 \leq J_1,J_2 \leq d/\lambda \\ (J_1,J_2) = 1 \\ (J_1J_2,d/\lambda) = 1}} \sg_{\lambda J_1}\sg_{\lambda J_2} \left(\sum_{\ss{1 \leq |r_1|,|r_2| <d/\lambda \\ (r_1r_2,d) = 1 \\ J_1r_1 \equiv J_2r_2 \pmod{d/\lambda}}} \frac{(-1)^{r_1+r_2}}{(r_1r_2)^2} + O\left(\frac{\lambda}{d}\right)\right) + O\left(\frac{1}{P^-(d)} S_f(x)^2\right) \nonumber \\
&= \frac{1}{4\pi^4} \sum_{\ss{\lambda | d \\ \lambda < d}} \sum_{\ss{1 \leq |r_1|,|r_2| <d/\lambda \\ (r_1r_2,d) = 1}} \frac{(-1)^{r_1+r_2}}{(r_1r_2)^2} \sum_{\ss{1 \leq J_1,J_2 \leq d/\lambda \\ (J_1,J_2) = 1 \\ (J_1J_2,d/\lambda) = 1 \\ J_1r_1 \equiv J_2r_2 \pmod{d/\lambda}}} \sg_{\lambda J_1}\sg_{\lambda J_2} + O\left(\frac{1}{P^-(d)} S_f(x)^2\right). \label{eq:noLargeRs}
\end{align}
Note that, fixing $1 \leq |r_1|,|r_2|, J_1 < d/\lambda$, any possible choice of $J_2$ is unique and fixed by the congruence in the inner sum. 
By hypothesis $\sg_{\lambda J_2} \ll \e S_f(x)$, and therefore the first expression in \eqref{eq:noLargeRs} is
$$
\ll \e S_f(x) \left(\sum_{\ss{1 \leq |r_1|,|r_2| < d \\ (r_1r_2,d) = 1}} \frac{1}{(r_1r_2)^2}\right) \sum_{\ss{\lambda|d \\ \lambda < d}} \sum_{\ss{1 \leq J \leq d/\lambda \\ (J,d/\lambda) = 1}} \sg_{\lambda J} \ll \e S_f(x)^2.
$$
It follows, therefore, that
$$
\Delta \ll \left(\e + \frac{1}{P^-(d)}\right)S_f(x)^2 \ll \e S_f(x)^2.
$$
By Chebyshev's inequality, we therefore deduce that for all but $O(\e^{1/2}\log^2(1/\e) d)$ choices of $1 \leq \ell \leq d-1$,
$$
\sum_{1 \leq j \leq d-1} \|\frac{\ell j}{d}\|^2 \sg_j = \left(\frac{1}{12} + O(\log(1/\e)^{-1})\right) S_f(x).
$$
Therefore, invoking \eqref{eq:distToFrac}, we deduce that for all but $O(\e^{1/2}\log^2(1/\e)d)$ choices of $1 \leq \ell \leq d-1$,
$$
\mb{D}(f^\ell,1;x)^2 \geq 8\cdot \left(\frac{1}{12} + O(\log(1/\e)^{-1}) \right) S_f(x) = \left(\frac{2}{3} + O(\log(1/\e)^{-1})\right) S_f(x),
$$
as claimed.
\end{proof}

\subsubsection{Case 2: some $\sg_{j_0}$ is large}
We next consider the case in which \emph{some} $\sg_{j_0}$ is large relative to $S_f(x)$. When $\mc{C}_d(\e)$ is large the following proposition will provide an alternative bound for the partial sums of $f^{\ell}$ when $\ell \in \mc{C}_d(\e)$.
\begin{prop}\label{prop:TKAppCes}
Let $d \geq 2$, $\e \in (0,1)$ and $3 \leq x \leq x_0$.
Let $c > 0$ satisfy $P^-(d) > 1/c$, and assume that $|\mc{C}_d(\e)| \geq cd$. Put $m := \lceil 2/c^2\rceil$ and suppose that
$$
m^2\log(1/\e) < \frac{1}{32}\log\log x.
$$
Then for any $1 \leq j, \ell \leq d-1$ and $\lambda \in (0,1)$ at least one of the following bounds hold:
$$
\frac{1}{x}\left|\sum_{n \leq x} f^\ell(n)\right| \leq \e,
$$
or else
$$
\frac{1}{x} \left|\sum_{n \leq x} f^\ell(n)\right| \ll \|\ell j/d\|^{-1}\left(\lambda \e^{-32m^2} + \log(e/\lambda)\left(\lambda^{1-2/\pi} + \frac{1}{\sg_j}\right) + \frac{1}{\sqrt{\sg_j}} + \frac{1}{(\log x)^{1-2/\pi+o(1)}}\right).
$$
\end{prop}
\begin{proof}
We may assume that $x$ is as large (and $\e$ as small) as desired, otherwise at least one of the alternatives is trivial. We may also assume that $\sg_j \geq 1$ and $\lambda \e^{-32m^2} < 1/2$, since otherwise the second alternative is trivial. \\
Suppose the first alternative fails. Arguing as in the proof of Proposition \ref{prop:approxHom},
$$
\mb{D}(f^\ell,n^{it_\ell};x)^2 \leq 2\log(1/\e) < \frac{1}{16m^2}\log\log x.
$$ 
for $\e$ small enough. \\
Seeking to apply \cite[Thm. 4]{GSDecay} below, we must introduce some notation. For each $1 \leq \ell \leq d-1$ define
$$
F_{\ell}(s) := \prod_{p \leq x} \left(1 + \sum_{j \geq 1} \frac{f^\ell(p^j)}{p^{js}}\right), \quad s \in \mb{C}.
$$
Let $y_{\ell,0} \in [-2\log x,2\log x]$ be chosen so that $|F_{\ell}(1+iy_{\ell,0})| = \max_{|y| \leq 2\log x} |F_{\ell}(1+iy)|$, and set
$$
y_{\ell} := \begin{cases} y_{\ell,0} &\text{ if } |y_{\ell,0}| < \tfrac{1}{2}\log x \\ 0 &\text{ otherwise.} \end{cases}
$$
We will need an upper bound for $|y_{\ell}|$ whenever $y_{\ell} \neq 0$, so assume for the moment that this is the case. By Mertens' theorem, given any $y \in \mb{R}$,
$$
\frac{|F_{\ell}(1+iy)|}{\log x} \asymp \frac{1}{\log x} \exp\left(\sum_{p \leq x} \frac{\text{Re}(f^{\ell}(p)p^{-iy})}{p}\right) \asymp e^{-\mb{D}(f,n^{iy};x)^2}.
$$
Thus, if $y_{\ell} \neq 0$ then
$$
\mb{D}(f^\ell,n^{iy_{\ell}};x)^2 = \min_{|y| \leq 2\log x} \mb{D}(f^{\ell},n^{iy};x)^2 + O(1) \leq \mb{D}(f^{\ell},n^{it_{\ell}};x)^2 + O(1) \leq 2 \log(1/\e) + O(1).
$$
By \eqref{eq:pretTriProd} and the crude bound $|t_{\ell}-y_{\ell}| \leq 2\log x$ we thus obtain
$$
\mb{D}(1,n^{i(t_\ell-y_\ell)};x) \leq \mb{D}(f^{\ell},n^{iy_{\ell}};x) + \mb{D}(f^{\ell}, n^{it_{\ell}};x) \leq 2\sqrt{2 \log(1/\e)} + O(1).
$$
In light of \eqref{eq:big1nit} and the hypothesis $\log(1/\e) < \frac{1}{32}\log\log x$, we see that $|t_{\ell}-y_{\ell}| \leq 100$. Thus, \eqref{eq:small1nit} delivers
$$
\log(1+|t_{\ell}-y_{\ell}|\log x) \leq 10\log(1/\e) + O(1).
$$
Using Proposition \ref{prop:approxHom}, we therefore conclude that
\begin{equation} \label{eq:ylBound}
|y_{\ell}| \leq |y_{\ell}-t_{\ell}| + |t_\ell| \ll \frac{\e^{-10} + \e^{-32m^2}}{\log x} \ll \frac{\e^{-32m^2}}{\log x},
\end{equation} 
a bound that we shall employ momentarily. \\
With this setup complete we may now proceed with the proof of the proposition. By \cite[Lem. 7.1]{GSDecay} we have
\begin{align}\label{eq:twistChiEll}
\frac{1}{x}\sum_{n \leq x} f^\ell(n) &= \frac{x^{iy_\ell}}{1+iy_\ell} \frac{1}{x} \sum_{n \leq x} f^\ell(n) n^{-iy_\ell} + O\left(\frac{e^{\mb{D}(f^\ell,n^{iy_\ell};x)\sqrt{(2+o(1))\log\log x}}}{\log x}\right) \nonumber\\
&= \frac{x^{iy_\ell}}{1+iy_\ell} \frac{1}{x} \sum_{n \leq x} f^\ell(n) n^{-iy_\ell} + O\left((\log x)^{\tfrac{1}{2\sqrt{2}}-1+o(1)}\right).
\end{align}
Let now $S_j := \{p \leq q : f(p) = e(j/d)\}$ and define the completely additive function 
$$
\Omega_{S_j}(n) := \sum_{\ss{p^k || n \\ p \in S_j}} k = \sum_{\ss{mp = n \\ p \in S_j}} 1,
$$
whose mean-value over $n \leq x$ is asymptotically
$$
\sum_{\ss{p^k \leq x \\ p \in S_j}} \frac{1}{p^k}\left(1-\frac{1}{p}\right) + O\left(\sqrt{\frac{\log\log x}{\log x}}\right) = \sg_j + O\left(\sqrt{\frac{\log\log x}{\log x}}\right).
$$
By the Tur\'{a}n-Kubilius inequality \cite[Thm. III.3.1]{Ten} and complete multiplicativity, we deduce that
\begin{align*}
\frac{1}{x}\sum_{n \leq x} f^\ell(n)n^{-iy_\ell} &= \frac{1}{x}\sum_{n \leq x} f^\ell(n)n^{-iy_\ell} \cdot \frac{\Omega_{S_j}(n) + (\sg_j-\Omega_{S_j}(n))}{\sg_j} \\
&= \frac{1}{\sg_j} \cdot \frac{1}{x}\sum_{\ss{mp \leq x \\ p \in S_j}} f^\ell(p)p^{-iy_\ell} f^\ell(m)m^{-iy_\ell} + O\left(\frac{1}{x\sqrt{\sg_j}} \sum_{n \leq x} \frac{|\Omega_{S_j}(n)-\sg_j|}{\sqrt{\sg_j}}\right) \\
&= \frac{1}{\sg_j} \sum_{\ss{p \leq x \\ p \in S_j}} \frac{f^\ell(p)p^{-iy_\ell}}{p} \frac{p}{x}\sum_{m \leq x/p} f^\ell(m)m^{-iy_\ell} + O\left(\frac{1}{\sqrt{\sg_j}}\right).
\end{align*}
We split the sum over $p\leq x$ into the segments $p \leq x^{\lambda}$ and $x^{\lambda} < p \leq x$. The contribution from the second segment to the above is
$$
\ll \frac{1}{\sg_j} \sum_{x^\lambda < p \leq x} \frac{1}{p} \ll \frac{\log(e/\lambda)}{\sg_j}.
$$
Now, for each $p \leq x^\lambda$ we apply the Lipschitz bound \cite[Thm. 4]{GSDecay}, which yields
$$
\frac{p}{x} \sum_{n \leq x/p} f^\ell(n) n^{-iy_\ell} = \frac{1}{x} \sum_{n \leq x} f^\ell(n) n^{-iy_\ell} + O\left(\lambda^{1-2/\pi} \log(e/\lambda) + \frac{1}{(\log x)^{1-2/\pi+o(1)}}\right).
$$
We thus conclude that
\begin{align*}
\frac{1}{x}\sum_{n \leq x} f^\ell(n)n^{-iy_\ell} &= \left(\frac{1}{\sg_j} \sum_{\ss{p \leq x^\lambda \\ p \in S_j}} \frac{f^\ell(p)p^{-iy_\ell}}{p}\right) \frac{1}{x} \sum_{n \leq x} f^\ell(n)n^{-iy_\ell} \\
&+ O\left(\frac{1}{\sqrt{\sg_j}} + \log(e/\lambda) \left(\frac{1}{\sg_j} + \lambda^{1-2/\pi}\right) + (\log x)^{\tfrac{2}{\pi}-1+o(1)}\right).
\end{align*}
For each $p \leq x^\lambda$, \eqref{eq:ylBound} yields
$$
p^{-iy_\ell} = 1 + O(\lambda \e^{-32m^2}),
$$
and therefore
\begin{align*}
\frac{1}{\sg_j} \sum_{\ss{p \leq x^\lambda \\ p \in S_j}} \frac{f^\ell(p)p^{-iy_\ell}}{p} &= \frac{1}{\sg_j} \sum_{\ss{p \leq x \\ p \in S_j}} \frac{f^\ell(p)}{p} + O\left(\lambda \e^{-32m^2} + \frac{\log(e/\lambda)}{\sg_j}\right) \\
&= e(\ell j/d) + O\left(\lambda \e^{-32m^2} + \frac{\log(e/\lambda)}{\sg_j}\right).
\end{align*}
Hence, appealing once again to \cite[Lem. 7.1]{GSDecay}, we find
\begin{align*}
\frac{1}{x} \sum_{n \leq x} f^\ell(n) = e(\ell j/d) \frac{1}{x}\sum_{n \leq x} f^\ell(n) + O\left(\frac{1}{\sqrt{\sg_j}} + \log(e/\lambda) \left(\frac{1}{\sg_j} + \lambda^{1-2/\pi}\right) + \lambda \e^{-32m^2}+ (\log x)^{\tfrac{2}{\pi}-1+o(1)}\right).
\end{align*}
Rearranging the above and using $|1-e(\ell j/d)| \gg \|\ell j /d\|$, the claim follows.
\end{proof}

\begin{proof}[Proof of Theorem \ref{thm:meanSqShortSumsParams}]
Let $\e > 0$ be a small parameter to be chosen later in terms of 
$$
\Sigma := \min\{P^-(d), 2+\sum_{\ss{p\leq x \\ f(p) \neq 0,1}} p^{-1}\},
$$
subject only to the constraints 
\begin{equation}\label{eq:SigConst}
\e \geq \Sigma^{-1/3} \geq P^-(d)^{-1/3}.
\end{equation}
We may assume in what follows that $\Sigma$ (and thus also $x$ and $d$) is larger than any specified constant, since otherwise the claim is trivial (by adjusting the implicit constant appropriately). \\
Set $M := \lceil 2/\e \rceil$. Note that as $\Sigma \leq \log\log x + O(1)$, the constraint \eqref{eq:SigConst} implies that when $x$ is large enough,
$$
32M^2 \log(1/\e)  \leq  \frac{512}{3} \Sigma^{2/3} \log \Sigma < \log\log x.
$$
As above, write
$$
\mc{C}_d(\sqrt{\e}) = \Big\{1 \leq \ell \leq d-1 : \frac{1}{x}\Big|\sum_{n \leq x} f^\ell(n) \Big| \geq \e^{1/2} \Big\}.
$$
We consider several cases.\\
\underline{Case 1:} Suppose first that $|\mc{C}_d(\sqrt{\e})| \leq \sqrt{\e} d$. Adding in the contribution from $\ell = 0$, we trivially have
$$
\mc{L} := \frac{1}{d}\sum_{0 \leq \ell \leq d-1} \left|\sum_{n \leq x} f^\ell(n)\right|^2 \leq \frac{|\mc{C}_d(\sqrt{\e})|+1}{d} x^2 + \sqrt{\e} x^2 \ll \sqrt{\e} x^2.
$$
\underline{Case 2:} Suppose next that $|\mc{C}_d(\sqrt{\e})| > \sqrt{\e} d$. For each $1 \leq \ell \leq d-1$ we choose $t_{\ell} \in [-2M/\e,2M/\e]$ so that
$$
\mb{D}(f^\ell, n^{it_\ell}; x) = \min_{|t| \leq 2M/\e} \mb{D}(f^\ell, n^{it};x).
$$
Since $\sqrt{\e} P^-(d) \geq P^-(d)^{5/6} > 1$, by Proposition \ref{prop:approxHom} we have 
\begin{equation}\label{eq:tlEst}
\max_{1 \leq \ell \leq d-1} |t_\ell| \leq \frac{3 \e^{-(20/\e)^2}}{\log x}.
\end{equation}
As above, let 
$$
S_j := \{p\leq q : f(p) = e(j/d)\}, \quad \sg_j := \sg_j(x) = \sum_{\ss{p \leq x \\ p \in S_j}} \frac{1}{p},
$$ 
and define
$$
S_f(x) := \sum_{1 \leq j \leq d-1} \sg_j.
$$
We consider two subcases. \\
\underline{Case 2.a)} Suppose there is $1 \leq j_0 \leq d-1$ such that 
\begin{equation}\label{eq:largesgjHyp}
\sg_{j_0} \geq \e S_f(x) \gg \e\Sigma.
\end{equation}
Take $J := \llf \e^{-1/2}\rrf$. 
By the Erd\H{o}s-Tur\'{a}n inequality \cite[Thm. I.6.15]{Ten} we have
\begin{align}\label{eq:ETApp}
|\, |\{1 \leq \ell \leq d : \|\ell j_0/d\| \leq \e^{1/2}\}| - 2\e^{1/2}d \, | &\ll \frac{d}{J} + d\sum_{1 \leq k \leq J} \frac{1}{k} \left|\sum_{1 \leq \ell \leq d} e\left(\frac{klj_0}{d}\right)\right| \nonumber \\
&\ll \e^{1/2} d + d^2\sum_{\ss{1 \leq k \leq J \\ d|kj_0}} \frac{1}{k}.
\end{align}
We observe that if $d|kj_0$, then as $j_0 < d$,
$$
J \leq 1/\e < P^-(d) \leq d/(j_0,d) \leq k. 
$$
Thus, the sum on the RHS of \eqref{eq:ETApp} is zero, and hence
$$
|\{1 \leq \ell \leq d : \|\ell j_0/d\| \leq \e^{1/2}\}| = 2\e^{1/2} d + O(\e^{1/2}d) \ll \e^{1/2}d.
$$
It follows that 
$$
\|\ell j_0/d\| > \e^{1/2} \text{ for all but } O\left(\e^{1/2} d\right) \text{ choices of } 1 \leq \ell \leq d-1.
$$
Now, let $\lambda \in (0,1)$ be a parameter to be chosen shortly, and let $\ell$ satisfy $\|\ell j_0/d\| > \e^{1/2}$. By Proposition \ref{prop:TKAppCes} (taking $c = \sqrt{\e}$ and $\sqrt{\e}$ in place of $\e$ there), 
either 
$$
\Big|\sum_{n \leq x} f^\ell(n)\Big| \leq \sqrt{\e} x,
$$ 
or else
\begin{align*}
\frac{1}{x}\left|\sum_{n \leq x} f^\ell(n)\right| \ll \e^{-1/2} \left(\lambda \e^{-(20/\e)^2} + \log(e/\lambda)\left(\lambda^{1-2/\pi} + \frac{1}{\sg_{j_0}}\right) + \frac{1}{\sqrt{\sg_{j_0}}}\right).
\end{align*}
We select
$$
\lambda := \min\{\e^{1+(20/\e)^2}, e^{-\sqrt{\sg_{j_0}}}\},
$$
so that
for all but $O(\e^{1/2}d)$ choices of $\ell$,
$$
\frac{1}{x}\left|\sum_{n \leq x} f^\ell(n)\right| \ll \e^{-1/2} \left(\e + \frac{1}{\sqrt{\sg_{j_0}}} + \frac{\log(1/\e)}{\e^2 \sg_{j_0}} + \sg_{j_0}^{1/2}e^{-\frac{1}{3}\sqrt{\sg_{j_0}}}\right) \ll \e^{1/2} + \frac{1}{\sqrt{\e\sg_{j_0}}} + \frac{\log(1/\e)}{\e^{5/2} \sg_{j_0}}.
$$
For the remaining $\ll \e^{1/2}d$ choices of $\ell$ we simply apply the trivial bound. Averaging over $0 \leq \ell \leq d-1$ and using \eqref{eq:largesgjHyp}, we thus obtain
$$
\mc{L} \ll \e^{1/2} x^2 + x^2\left(\e + \frac{1}{\e\sg_{j_0}} + \frac{\log(1/\e)^2}{\e^5 \sg_{j_0}^2}\right) \ll \left(\e^{1/2} + \frac{1}{\e^2 \Sigma} + \frac{\log(1/\e)^2}{\e^{7}\Sigma^2}\right)x^2.
$$
\underline{Case 2.b)} Suppose finally that $\sg_r < \e S_f(x)$ for all $1\leq r \leq d-1$. \\
By Lemma \ref{lem:SmallSigCase}, we find that for all but $O(\e^{1/2}\log^2(1/\e)d)$ choices of $\ell$ we have
$$
\mb{D}(f^\ell, 1;x)^2 \geq \left(\frac{2}{3} + O(\log(1/\e)^{-1})\right) \Sigma.
$$
By \eqref{eq:pretTri} we also have
\begin{equation}\label{eq:fellto1}
\mb{D}(f^\ell,1;x) \leq \mb{D}(f^\ell,n^{it_\ell};x) + \mb{D}(1,n^{it_\ell};x),
\end{equation}
and moreover by \eqref{eq:tlEst},
\begin{equation}\label{eq:nitlDist}
\mb{D}(1,n^{it_\ell};x)^2 = \log(1+|t_\ell|\log x) + O(1) \leq \frac{400}{\e^2}\log(1/\e) + O(1).
\end{equation}
Since $\e \geq \Sigma^{-1/3}$ by assumption, we have
$$
\frac{400}{\e^2} \log(1/\e) \leq 400 \Sigma^{2/3} \log \Sigma \leq \frac{1}{20} \Sigma
$$
whenever $\Sigma$ is large enough. Combining this with \eqref{eq:nitlDist}, we thus find that $\mb{D}(1,n^{it_\ell};x)^2 \leq \Sigma/20$.
Inserting this into \eqref{eq:fellto1} and using the inequality $(a+b)^2 \leq 2(a^2+b^2)$ we deduce that
\begin{align*}
\mb{D}(f^\ell,n^{it_\ell};x)^2 \geq \frac{1}{2}\mb{D}(f^\ell,1;x)^2 - \mb{D}(1,n^{it_\ell};x)^2 &\geq \left(\frac{1}{3} + O(\log(1/\e)^{-1})\right) \Sigma - \frac{400}{\e^2}\log(1/\e) \\
&\geq \left(\frac{1}{4} + O(\log(1/\e)^{-1})\right) \Sigma,
\end{align*}
Therefore, on applying \eqref{eq:HMT} we obtain
$$
\frac{1}{x}\left|\sum_{n \leq x} f^\ell(n)\right| \ll \frac{\e}{M} + \frac{1}{\log x} + \mb{D}(f^\ell,n^{it_\ell};x)^2 e^{-\mb{D}(f^\ell,n^{it_\ell};x)^2} \ll \e^2 + \Sigma e^{-\frac{1}{5}\Sigma},
$$
all but $O(\e^{1/2}\log^2(1/\e) d)$ choices of $\ell$.
Bounding trivially the partial sums corresponding to the exceptional $\ell$, we obtain
$$
\mc{L} \ll x^2\left(\e^{1/2}\log^2(1/\e) + \Sigma^2e^{-\frac{2}{5}\Sigma}\right)
$$
in this case. \\
\underline{Conclusion:} Combining Cases 1, 2(a) and 2(b) together, 
$$
\mc{L} \ll x^2\left(\e^{1/2}\log^2(1/\e) + \frac{1}{\e^2 \Sigma} + \frac{\log^2(1/\e)}{\e^{7}\Sigma^2} + \Sigma^2 e^{-\frac{2}{5}\Sigma}\right) \ll x^2\left(\e^{1/2}\log^2(1/\e) + \frac{1}{\e^2\Sigma} + \frac{\log^2(1/\e)}{\e^{7}\Sigma^2}\right).
$$
We choose $\e = \Sigma^{-4/15}$, which yields the bound
$$
\mc{L}
\ll x^2 \frac{(\log \Sigma)^2}{\Sigma^{2/15}},
$$ 
as claimed.
\end{proof}

\section{Proof of Theorem \ref{thm:HilApp}} \label{sec:HilAppLow}
\noindent We now turn to the proof of Theorem \ref{thm:HilApp}, concerning lower bounds for the sum
$$
\sum_{\ss{p \leq x \\ f(p) \neq 0,1}} \frac{1}{p}, \text{ with } f \in \mc{M}_1(x_0;d), \, x \geq x_0^{\delta}.
$$
We will analyze the \emph{prime} values $(f(p))_{p \leq x}$ using the obvious implication 
$$
f(p) = 1 \text{ for all } p |n \Rightarrow f(n) = 1
$$ 
and the (weak) equidistribution property of the \emph{integer} values $(f(n))_{n \leq x_0}$. \\
Our general argument will allow us to handle $d \ll (\eta \delta)^{-1} e^{(\log x_0)^c}$ for some absolute constant $c > 0$. Using zero-density estimates for Dirichlet $L$-functions, we may extend this range when $f$ is a character of order $d$ as follows.
\begin{prop}\label{prop:HilAppChar}
Fix $c \in (0,1)$. If $\chi$ is a primitive character modulo prime $q$ with order $d \geq e^{(\log q)^c}$ then
$$
\sum_{\ss{p \leq q \\ \chi(p) \neq 1}} \frac{1}{p} \geq (c-o(1)) \log\log q, \quad q \ra \infty.
$$
\end{prop}
\begin{proof}
We will estimate the product
\begin{equation}\label{eq:L1ProdEst}
\prod_{1 \leq \ell \leq d} |L(1+1/\log q,\chi^\ell)|
\end{equation}
in two ways. \\
For the first, note that if $\chi^\ell$ is a non-exceptional, non-principal character (see e.g. \cite[Sec. 4.2]{GraMan} for a precise definition) then by e.g. \cite[(4.3)]{GraMan}, 
$$
|L(1+1/\log q,\chi^\ell)| = \exp\left(\text{Re}\left(-\sum_{p \leq q} \log\Big( 1-\frac{\chi^\ell(p)}{p^{1+1/\log q}}\Big) + O(1)\right)\right) = \exp\left(\sum_{p \leq q} \frac{\chi^\ell(p)}{p} + O(1)\right).
$$
This estimate also holds when $\ell = d$ using properties of the Riemann zeta function. This bound thus holds for all but at most one $1 \leq \ell \leq d-1$, and if such an exception exists then the corresponding character must be real.  \\
If $d$ is even then let $\ell_0 \in \{1,\ldots,d-1\}$ be such that $\chi^{\ell_0}$ is real (there is a unique such choice); otherwise, choose $\ell_0$ arbitrarily. Whether $\chi^{\ell_0}$ is exceptional or not, \cite[Thms. II.8.20-II.8.21]{Ten} implies that if $q$ is sufficiently large then
$$
|L(1+1/\log q, \chi^{\ell_0})| \prod_{p \leq q} \left(1-\frac{\chi^{\ell_0}(p)}{p}\right) \gg \frac{1}{\sqrt{q}\log q}.
$$
Therefore, by the preceding estimates and orthogonality,
\begin{equation}\label{eq:LowBdProdL1s}
\prod_{1 \leq \ell \leq d} |L(1+1/\log q,\chi^\ell)| \geq \frac{1}{\sqrt{q}\log q} \exp\left(\sum_{1 \leq \ell \leq d} \left(\sum_{p \leq q} \frac{\chi^\ell(p)}{p} + O(1)\right)\right) = \frac{1}{\sqrt{q}\log q}\exp\left(d\left(\sum_{\ss{p \leq q \\ \chi(p) = 1}} \frac{1}{p} + O(1)\right)\right).
\end{equation}
For the second estimation of \eqref{eq:L1ProdEst} we invoke zero-density estimates. Let $\theta \in (0,c)$ be fixed, and set 
$$
\sg_0 := 1-\frac{1}{(\log q)^{1-\theta}}, \quad T := (\log q)^3.
$$ 
We define
\begin{align*}
\mc{G} &:= \mc{G}(\sg_0,T) := \{\chi^\ell\}_{1\leq \ell \leq d-1} \cap \{\psi \pmod{q} : L(s,\psi) \neq 0 \text{ for all } \text{Re}(s) \in (sg_0,1], \, |\text{Im}(s)| \leq T\} \\
\mc{B} &:= \mc{B}(\sg_0,T) := \{\chi^\ell\}_{1 \leq \ell \leq d} \bk \mc{G}.
\end{align*}
If $1 \leq \ell \leq d-1$ then by e.g. \cite[Thm. II.8.18]{Ten},
$$
|L(1+ 1/\log q,\chi^\ell)| \ll \log q.
$$
When $\ell = d$ this bound is also valid since by Mertens' theorem,
$$
|L(1+1/\log q, \chi_0)| \ll \exp\left(\sum_p \frac{1}{p^{1+1/\log q}}\right) \ll \exp\left(\sum_{p \leq q} \frac{1}{p}\right) \ll \log q.
$$
Therefore, there is an absolute constant $C > 0$ such that 
\begin{align} \label{eq:uppBdProdL1s}
\prod_{1 \leq \ell \leq d} |L(1+1/\log q,\chi^\ell)| \leq \prod_{\ss{1 \leq \ell \leq d \\ \chi^\ell \in \mc{G}}} |L(1+1/\log q, \chi^\ell)| \cdot (C \log q)^{|\mc{B}|}.
\end{align}
Since $(1-\sg_0)\log q \gg 1$, replacing $1$ with $1 + 1/\log q$ and following the proof of \cite[Lem. 5.4]{LamMan} \emph{mutatis mutandis}, we find that  for each $\chi^\ell \in \mc{G}$ and any $X \geq 2$,
$$
\log L(1+1/\log q,\chi^\ell)  =  -\sum_{p \leq X} \log\left(1-\frac{\chi^\ell(p)}{p^{1+1/\log q}}\right)	+ O\left(\frac{1}{\log q} + (\log q)^4 e^{-\frac{(\log q)^\theta}{2} \frac{\log X}{\log q}}\right).
$$
Taking $X := \exp(10(\log q)^{1-\theta} \log\log q)$ and bounding the prime sum trivially, we obtain
$$
\log |L(1+1/\log q,\chi^\ell)| \leq \log\log X + O(1) = (1-\theta + o(1))\log\log q,
$$
for each $\chi^\ell \in \mc{G}$. Furthermore, increasing $C$ if needed, the log-free zero density estimate \cite[(18.9)]{IK} shows that 
$$
|\mc{B}| \ll (qT)^{C(1-\sg_0)} \leq e^{2C(\log q)^{\theta}}.
$$
As $\theta < c$, we obtain $|\mc{B}| = o(d)$ when $q$ is sufficiently large.
Invoking both of these estimates in \eqref{eq:uppBdProdL1s}, we deduce that
$$
\prod_{1 \leq \ell \leq d} |L(1+1/\log q,\chi^\ell)| \leq (C \log q)^{(1-\theta + o(1))(d-|\mc{B}|)} (C \log q)^{|\mc{B}|} \leq \left(C (\log q)^{1-\theta + \theta \tfrac{|\mc{B}|}{d} + o(1)}\right)^d = (\log q)^{d(1-\theta+o(1))}.
$$
Comparing this with \eqref{eq:LowBdProdL1s}, we find
\begin{align*}
\frac{1}{\sqrt{q}\log q} \exp\left(d\left(\sum_{\ss{p \leq q \\ \chi(p) = 1}} \frac{1}{p} + O(1)\right)\right) &\leq \prod_{1 \leq \ell \leq d} |L(1+1/\log q,\chi^\ell)| \\
&\leq (\log q)^{d(1-\theta+o(1))} = \exp\left(d (1-\theta + o(1))\sum_{p \leq q} \frac{1}{p}\right).
\end{align*}
Since $(\sqrt{q}\log q)^{1/d} \ll 1$, on taking $d$th roots and rearranging we obtain
$$
\sum_{\ss{p \leq q \\ \chi(p)\neq 1}} \frac{1}{p} = \sum_{p \leq q} \frac{1}{p} - \sum_{\ss{p \leq q \\ \chi(p) = 1}} \frac{1}{p} \geq (\theta - o(1)) \log\log q.
$$
As this bound holds for every $\theta \in (0,c)$ the claim follows. 
\end{proof}
\begin{proof}[Proof of Theorem \ref{thm:HilApp}]
The conclusion is strongest when $x = x_0^\delta$, so we assume this in what follows. \\
(a) Write 
$$
E_{\neq 0, 1}(x) := \sum_{\ss{p \leq x \\ f(p) \neq 0,1}} \frac{1}{p},
$$ 
and as above set
$$
c_f = c_f(x_0) := \prod_{\ss{p \leq x_0 \\ f(p) = 0}} \left(1-\frac{1}{p}\right).
$$
Since $f$ weakly equidistributes at scale $x_0$, 
\begin{equation}\label{eq:levelsetUpp}
|\{n \leq x_0 : f(n) = 1\}| \leq 100c_f\frac{x_0}{d}.
\end{equation}
We now derive a corresponding lower bound. Let $g$ be the non-negative completely multiplicative function defined at primes by
$$
g(p) := \begin{cases} 1 \text{ if } p \leq x \text{ and } f(p) = 1 \\ 0 \text{ otherwise.} \end{cases}
$$
We then observe that
\begin{equation}\label{eq:levelsetLow}
|\{n \leq x_0 : f(n) = 1\}| \geq \sum_{n \leq x_0} g(n).
\end{equation}
Now if $x_0$ is large enough then by \cite[Thm. 2]{HilMV} there are absolute constants $A,\beta > 0$ such that 
\begin{equation}\label{eq:HilBd}
\sum_{n \leq x_0} g(n) \geq A x_0\prod_{\ss{p \leq x_0 \\ f(p) = 0}} \left(1-\frac{1}{p}\right) \cdot \prod_{\ss{p \leq x_0 \\ f(p) \neq 0}}\left(1-\frac{1}{p}\right)\left(1-\frac{1_{f(p) = 1} 1_{p \leq x}}{p}\right)^{-1} \cdot \left(\sg_-\Big(e^{\mb{D}(g,1;x_0)^2}\Big) + O(e^{-(\log x_0)^\beta})\right),
\end{equation}
where $\sg_-(u) := u \rho(u)$. 
Evaluating the various factors in this expression, we see that 
\begin{align*}
&x_0\prod_{\ss{p \leq x_0 \\ f(p) = 0}} \left(1-\frac{1}{p}\right)\cdot \prod_{\ss{p \leq x_0 \\ f(p) \neq 0}}\left(1-\frac{1}{p}\right)\left(1-\frac{1_{f(p) = 1} 1_{p \leq x}}{p}\right)^{-1} \\
&\gg c_f x_0\exp\left(-\sum_{\ss{p \leq x \\ f(p) \neq 0}} \frac{1-1_{f(p) = 1}}{p} - \sum_{x < p \leq x_0} \frac{1}{p} \right) \gg c_f x_0 \delta e^{-E_{\neq 0, 1}(x)}, 
\end{align*}
and also
\begin{align*}
e^{\mb{D}(g,1;x_0)^2} = \exp\left(\sum_{\ss{p \leq x_0 \\ f(p) = 0}} \frac{1}{p} + \sum_{x < p \leq x_0} \frac{1}{p} + E_{\neq 0,1}(x)\right) \ll (c_f \delta)^{-1}e^{E_{\neq 0, 1}(x)}.
\end{align*}
We set $c_3 = \beta$, take $C_3 > 0$ to be a large constant and assume that
\begin{equation}\label{eq:dUppbdCons}
d \leq C_3(\eta \delta)^{-1} e^{(\log x_0)^{c_3}}.
\end{equation}
Suppose for the sake of contradiction
that $E_{\neq 0,1}(x) < \log(1/\eta)$. Since $\sg_-(u)$ is a decreasing function of $u$, there is an absolute constant $C' > 0$ such that
$$
\sg_-(e^{\mb{D}(1,g;x_0)^2}) \geq \sg_-\left(c_f^{-1}\frac{C'}{\eta \delta}\right).
$$
Thus on combining \eqref{eq:levelsetUpp}, \eqref{eq:levelsetLow} and \eqref{eq:HilBd} we find that for some absolute constant $B > 0$,
\begin{equation}\label{eq:lowBdgSum}
c_f \frac{x_0}{d} \geq \frac{1}{100}\sum_{n \leq x_0} g(n) \geq B c_f x_0 \eta \delta  \left(\sg_-\left(c_f^{-1}\frac{C'}{\eta \delta}\right) + O(e^{-(\log x_0)^\beta})\right).
\end{equation}
If $C_3$ is large enough relative to $B$, 
the error term in \eqref{eq:lowBdgSum} will contribute $\leq \tfrac{1}{2}c_f x_0/d$ to the right-hand side.
Thus, rearranging \eqref{eq:lowBdgSum} and suitably modifying the implicit constant, this term may be deleted from \eqref{eq:lowBdgSum}. Incorporating the definition of $\sg_-(u)$,
we thus get 
$$
c_f \frac{x_0}{d}
\gg c_f x_0 \eta \delta \cdot c_f^{-1}\frac{1}{\eta \delta} \rho\left(c_f^{-1} \frac{C'}{\eta \delta}\right) \gg x_0 \rho\left(c_f^{-1} \frac{C'}{\eta\delta}\right).
$$
Since $f \in \mc{M}_1(x_0;d)$ by assumption we have $c_f \gg 1$, whence
$$
d \ll \rho(\tfrac{C''}{\eta \delta})^{-1}
$$
for some absolute $C'' > 0$. Choosing $C_1,C_2 > 0$ large enough, we obtain a contradiction whenever $d \geq C_1\rho\Big(\tfrac{C_2}{\eta\delta}\Big)$, and part (a) follows. \\
(b) Assume next that $d \geq C_3(\eta \delta)^{-1} e^{(\log x)^{\beta}}$ and $f = \chi$ a primitive character modulo prime $q$, of order $d$ (taking $x_0 = q^{3/2+\e}$, say). If $q$ is large enough then for any $c \in (0,\beta)$ we have $d \geq e^{(\log q)^c}$. Thus, by Proposition \ref{prop:HilAppChar},
$$
\sum_{\ss{p \leq x \\ \chi(p) \neq 0,1}} \frac{1}{p} \geq \sum_{\ss{p \leq q \\ \chi(p) \neq 1}} \frac{1}{p} - \sum_{x < p \leq q} \frac{1}{p} - \sum_{p : \chi(p) = 0} \frac{1}{p} \geq (\beta - o(1)) \log\log q - \log(1/\delta).
$$
Since $\log(1/(\eta\delta)) \leq c\log\log q$ for some $0 < c < \beta$, once $q$ is large enough we have
$$
\geq \log(1/\eta) + \Big( (\beta - o(1)) \log\log q - \log(1/\eta \delta)\Big) \geq \log(1/\eta),
$$
and the claim follows.
\end{proof}

\begin{proof}[Proof of Theorem \ref{thm:meanSqShortSums}]
We may assume that $d$ is larger (and thus $\delta$ is smaller) than any fixed constant, otherwise $G(d) \ll 1$ and the claimed bound is trivial. \\
Taking $x_0 = q^{3/2+\e}$ we have $\chi \in \mc{M}_1(x_0;d)$. Let $\eta = \delta$ and $0 < c_2 < c_3/2$, so that $\eta\delta > (\log q)^{-c}$ for some $c \in (0,c_3)$. Since $1/\delta^2 \ll c_2^2\frac{\log d}{\log\log d}$ and taking $c_2 > 0$ smaller if needed, by \eqref{eq:DickmanLB} we find
$$
d \geq C_1 \rho(C_2/(\eta\delta')), \quad \delta' := 2\delta/3.
$$
As in the proof of Theorem \ref{thm:levelSet} we have
$$
\log(1/\eta) \gg \log\log d,
$$ 
and so by Theorem \ref{thm:HilApp}(b),
$$
\sum_{\ss{p \leq x \\ \chi(p) \neq 0,1}} \frac{1}{p} \geq \log(1/\eta)  \gg \log\log d, \quad x > q^{\delta}.
$$
We thus conclude that
$$
\Sigma = \min\Big\{P^-(d), 2 + \sum_{\ss{p \leq x \\ \chi(p) \neq 0,1}} p^{-1}\Big\} \gg \min \{P^-(d), \log\log(ed)\} = G(d).
$$
The theorem now follows from Theorem \ref{thm:meanSqShortSumsParams}.
\end{proof}

\section{Improvements to P\'{o}lya-Vinogradov on Average} \label{sec:PVImp}
Let $\e > 0$ and let $\chi$ be a primitive character of order $d$ modulo a prime $q$. We define
$$
\Xi_d(\e) := \{\ell \in \mb{Z}/d\mb{Z} : \, \ell \neq 0, \, M(\chi^\ell) > \e \sqrt{q}\log q\}.
$$
To prove Theorem \ref{thm:PVImpAvg} we will show, using work of Granville and Soundararajan \cite{GSPret}, that the iterated sumsets of $\Xi_d(\e)$ satisfy rigid conditions; see Lemma \ref{lem:sumFree} below. When $d$ has no small prime factors this rigidity puts limits on the size of $\Xi_d(\e)$.
\begin{prop}\label{prop:smallXiD}
Let $q \geq 10$ be large. Then there is an absolute constant $C > 0$ such that if $k \geq 1$ and $\e > 0$ satisfy
$$
k \leq \frac{\log\log q}{10\log\log\log q}, \quad \e \geq C (\log q)^{-\tfrac{1}{3(2k+1)^2}}
$$
then
$$
|\Xi_d(\e)| \ll d\left(\frac{1}{k} + \frac{1}{P^-(d)}\right).
$$
\end{prop}
To prove Proposition \ref{prop:smallXiD} we need the following (slight) extension of \cite[Thm. 2]{GSPret} that gives a precise dependence of the bound on the number of characters involved.
\begin{lem}\label{lem:GSLargeG}
Let $q \geq 10$ be large and let $3 \leq g \leq \frac{\log \log q}{\log\log\log q}$ be odd. Then there is an absolute constant $C_0 > 0$ such that the following holds. \\
Let $\chi_1,\ldots,\chi_g$ be primitive characters with respective conductors $q_j \leq q$ and for which 
$$
M(\chi_j) > \sqrt{q_j}(\log q_j)^{1-1/g} \text{ for all } 1 \leq j \leq g.
$$ 
Suppose in addition that $\chi_1\cdots \chi_g$ is principal. Then as $q \ra \infty$,
$$
\prod_{1 \leq j \leq g} \frac{M(\chi_j)}{\sqrt{q_j}\log q} \leq C_0^g (\log q)^{-\tfrac{1+o(1)}{2g}}.
$$
\end{lem}
\begin{proof}
By \cite[Prop. 3]{GraMan}, there is an absolute constant $A > 0$ such that for each $1 \leq j \leq g$ there are: 
\begin{enumerate}[(i)]
\item a primitive character $\xi_j \pmod{m_j}$ with $m_j \leq (\log q)^{\tfrac{2}{g}}(\log\log q)^4$ and $\xi_j(-1) = -\chi_j(-1)$; and 
\item a scale $1 \leq N_j \leq q_j$, such that
$$
M(\chi_j) \leq A \frac{\sqrt{q_jm_j}}{\phi(m_j)} \left|\sum_{n \leq N_j} \frac{\chi_j\bar{\xi}_j(n)}{n}\right|.
$$
\end{enumerate}
By \cite[Lem. 4.3]{GSPret}, there is an absolute constant $B > 0$ such that
$$
\left|\sum_{n \leq N_j} \frac{\chi_j\bar{\xi}_j(n)}{n}\right| \leq B (\log N_j) e^{-\frac{1}{2}\mb{D}(\chi_j,\xi_j;N_j)^2} \leq B (\log q)e^{-\frac{1}{2}\mb{D}(\chi_j,\xi_j;q)^2},
$$
using in the last step the fact that for any sequence $\{f(p)\}_p \subset \mb{U}$ the map
$$
t \mapsto \sum_{p \leq t} \frac{2-(1-\text{Re}(f(p)))}{p}
$$
is non-decreasing. Choosing $D \geq 1$ absolute so that $\sqrt{m} \leq D\phi(m)$ for all $m \geq 1$ and then setting $C_0 := ABD$, we deduce that
$$
\frac{M(\chi_j)}{\sqrt{q_j}\log q} \leq C_0 e^{-\frac{1}{2}\mb{D}(\chi_j,\xi_j;q)^2}.
$$
Taking the product over $1 \leq j \leq g$ and using the inequality
$$
\sum_{1 \leq j \leq g} \mb{D}(\chi_j,\xi_j;q)^2 \geq \frac{1}{g}\mb{D}(1,\xi_1\cdots \xi_g;q)^2
$$
from the proof of \cite[Lem. 3.3]{GSPret}, we thus obtain
\begin{equation}\label{eq:prodMchis}
\prod_{1 \leq j \leq g} \frac{M(\chi_j)}{\sqrt{q_j}\log q} \leq C_0^g e^{-\frac{1}{2g}\mb{D}(1,\xi_1\cdots \xi_g;q)^2}.
\end{equation}
Set now $\xi := \xi_1\cdots \xi_g$, which is non-trivial since $\xi(-1) = (-1)^g \chi_1\cdots \chi_g(-1) = -1$. The conductor of $\xi$ is 
$$
M := [m_1,\ldots,m_g] \leq m_1\cdots m_g \leq (\log q)^2 (\log\log q)^{4g} \leq (\log q)^6,
$$
by our hypothesis on $g$. Thus, by the Siegel-Walfisz theorem (as in e.g. \cite[(3.1)]{GSPret}) we obtain
\begin{align*}
\mb{D}(1,\xi;q)^2 &\geq \sum_{M < p \leq q} \frac{1-\text{Re}(\xi(p))}{p} \geq \sum_{\ss{a \pmod{M} \\ (a,M) = 1}} (1-\text{Re}(\xi(a))) \sum_{\ss{M < p \leq q \\ p \equiv a \pmod{M}}} \frac{1}{p}  \\
&= \frac{1+o(1)}{\phi(M)} \left(\phi(M) - \text{Re}\left(\sum_{a \pmod{M}} \xi(a)\right)\right)\log\log q \\
&= (1+o(1))\log\log q.
\end{align*}
We hence deduce from \eqref{eq:prodMchis} that as $q \ra \infty$,
$$
\prod_{1 \leq j \leq g} \frac{M(\chi_j)}{\sqrt{q_j}\log q} \leq C_0^g (\log q)^{-\tfrac{1+o(1)}{2g}},
$$
as claimed.
\end{proof}
 
\begin{lem} \label{lem:sumFree}
Let $q \geq 10$ be large. Then there is an absolute constant $C \geq 1$ such that whenever $1 \leq k \leq \tfrac{\log \log q}{10\log\log\log q}$ and $C (\log q)^{-\tfrac{1}{3(2k+1)^2}} < \e < 1$,
$$
2k \Xi_d(\e) \cap \Xi_d(\e) = \emptyset.
$$
\end{lem}
\begin{proof}
For ease of notation, write $A := \Xi_d(\e)$. Suppose for the sake of contradiction that $a \in 2kA \cap A$. Then we can find $a_1,\ldots,a_{2k} \in A$ such that
\begin{equation}\label{eq:aRep}
a_1 + \cdots + a_{2k} - a \equiv 0 \pmod{d}.
\end{equation}
For convenience, write $a_{2k+1} \equiv -a \pmod{d}$, noting that $M(\chi^{a_{2k+1}}) = M(\chi^a)$. \\
Since $0 \notin A$ and $q$ is prime, $\chi^{a_j}$ is primitive for all $1 \leq j \leq 2k+1$. Also, by \eqref{eq:aRep}, 
$\chi^{a_1} \cdots \chi^{a_{2k+1}}$ is principal. 
Finally, since $\e \log q \geq (\log q)^{1-1/(2k+1)}$, Lemma \ref{lem:GSLargeG} yields
$$
\prod_{1 \leq j \leq 2k+1} \frac{M(\chi^{a_j})}{\sqrt{q}\log q} \leq C_0^{2k+1} (\log q)^{-\tfrac{1+o(1)}{2(2k+1)}}.
$$
On the other hand, since $a_j \in A$ for all $1 \leq j \leq 2k+1$ (as $A$ is symmetric), we obtain
$$
\prod_{1 \leq j \leq 2k+1} \frac{M(\chi^{a_j})}{\sqrt{q}\log q} > \e^{2k+1} \geq C^{2k+1} (\log q)^{-\tfrac{1}{3(2k+1)}}.
$$
This gives a contradiction whenever $C \geq C_0$ and $q$ is large enough, and the claim follows.
\end{proof}

Let $G$ be a finite Abelian group, written additively. For $k,\ell\geq 1$ we say that $A \subseteq G$ is a \emph{$(k,\ell)$-set} if $kA \cap \ell A = \emptyset$. Lemma \ref{lem:sumFree} shows that, under the claimed constraints on $k$ and $\e$,  $\Xi_d(\e)$ is a $(2k,1)$-set. It is clear that if $B \subseteq A$ and $A$ is a $(k,\ell)$-set then $B$ is also a $(k,\ell)$-set, so that $(k,\ell)$-sets form a partially-ordered set under inclusion. We say that $A$ is a \emph{maximal $(k,\ell)$-set} if $A$ is maximal with respect to this partial order.
\begin{thm1}[Bajnok \cite{Baj}, Thm. 3; Hamidoune-Plagne \cite{HamPla}, Thm. 2.4]
Let $k,\ell, n \geq 1$ with $k \neq \ell$. Suppose $G = \mb{Z}/n\mb{Z}$, and let $A \subseteq G$ be a maximal $(k,\ell)$-set. Then
$$
|A| \leq \max_{f|n} \frac{n}{f} \left(1+ \llf \frac{f-2}{k+\ell}\rrf\right).
$$
\end{thm1}

\begin{proof}[Proof of Proposition \ref{prop:smallXiD}]
By Lemma \ref{lem:sumFree}, $\Xi_d(\e)$ is a $(2k,1)$ set, and thus must be contained inside of a maximal $(2k,1)$-set. By the Bajnok-Hamidoune-Plagne theorem,
$$
|\Xi_d(\e)| \leq \max_{f|d} \frac{d}{f}\left(1+ \llf \frac{f-2}{2k+1}\rrf\right).
$$
When $f = 1$ the term inside the floor function here is $-1/(2k+1) \in [-1,0)$, and so the expression is $0$. Thus,
$$
|\Xi_d(\e)| \leq d \max_{\ss{f|d \\ f > 1}}\left(\frac{1}{f} + \frac{1}{2k+1}\right).
$$
The right-hand side is maximized at the smallest divisor of $d$ greater than 1, which is $P^-(d)$. The claim thus follows.
\end{proof}

\begin{proof}[Proof of Theorem \ref{thm:PVImpAvg}]
Let $q \geq 10$ be large, and let $1 \leq k \leq \tfrac{\log\log q}{10\log\log \log q}$ be a parameter to be chosen later. Setting $\e := C(\log q)^{-\tfrac{1}{3(2k+1)^2}}$ and splitting the sum over $j$ according to whether or not $j \in \Xi_d(\e)$, Proposition \ref{prop:smallXiD} implies that 
$$
\frac{1}{d}\sum_{1 \leq \ell \leq d-1} M(\chi^\ell) \ll (\sqrt{q}\log q)\left(\e + \frac{|\Xi_d(\e)|}{d}\right) \ll (\sqrt{q}\log q) \left(\e + \frac{1}{k} + \frac{1}{P^-(d)}\right).
$$
If we set $k := \llf \frac{1}{10}\sqrt{\frac{\log\log q}{\log\log \log q}}\rrf$ then 
$\e 
\leq C (\log \log q)^{-\tfrac{1}{2}}$, and the claimed bound follows.
\end{proof}

\begin{proof}[Proof of Corollary \ref{cor:unifPVImp}]
Note that since $d\geq 3$ is prime, $\chi^\ell$ has odd order $d$ for all $d \nmid \ell$. Applying Lemma \ref{lem:GSLargeG} with $\chi_j = \chi^\ell$ for all $1 \leq j \leq d$, we obtain\footnote{Note that this bound could be trivial if $d \gg (\log\log q)^{1/2}$.} (whether or not $M(\chi^\ell) \leq \sqrt{q}(\log q)^{1-1/d}$)
$$
M(\chi^\ell) = \left(\prod_{1 \leq j \leq d} \frac{M(\chi_j)}{\sqrt{q}\log q}\right)^{\frac{1}{d}} \sqrt{q}\log q \leq C_0(\sqrt{q}\log q) (\log q)^{-\tfrac{1+o(1)}{2d^2}}
$$
for each $1 \leq \ell \leq d-1$. Combined with Theorem \ref{thm:PVImpAvg}, as $q \ra \infty$ this gives
$$
\frac{1}{d}\sum_{1 \leq \ell \leq d-1} M(\chi^\ell) \ll (\sqrt{q}\log q)\min\left\{(\log q)^{-\tfrac{1}{3d^2}}, \left(\frac{\log\log\log q}{\log \log q}\right)^{1/2} + \frac{1}{d}\right\}.
$$
The transition point in these bounds occurs for $d \asymp \sqrt{\frac{\log\log q}{\log\log\log q}}$, and with a suitable implicit constant the claimed upper bound follows.
\end{proof}
\begin{rem}
Let $K := \llf \sqrt{\frac{\log \log q}{100\log(1/\e)}}\rrf$, and for each $1 \leq k \leq K$ let $A_k \subseteq \mb{Z}/d\mb{Z}$ be a maximal $(2k,1)$-set that contains $\Xi_d(\e)$. The Bajnok-Hamidoune-Plagne theorem stated above establishes an upper bound for each $|A_k|$, and therefore also for the rightmost expression in the chain of inequalities
\begin{equation}\label{eq:chainIneq}
|\Xi_d(\e)| \leq |\bigcap_{1 \leq k \leq K} A_k| \leq \min_{1 \leq k \leq K} |A_k|.
\end{equation}
It is reasonable to ask whether there could be a substantial difference in size between the intersection of the $A_k$ and the minimal $|A_k|$. However, note that any symmetric set $A$ satisfies $(2k) A \subseteq (2k+2)A$, as we can express any $m$-fold sum of terms in $A$ as an $(m+2)$-fold sum via
$$
a_1 + \cdots + a_k = a_1 + \cdots + a_k + a' + (-a')
$$
for any $a' \in A$. It follows that any $(2(k+1),1)$-set is also a $(2k,1)$-set, and it is possible therefore that $A_{k+1} \subseteq A_k$, i.e., the sums are nested. In such an event, the latter inequality in \eqref{eq:chainIneq} is sharp.
\end{rem}

\bibliographystyle{plain}
\bibliography{LargeOrderBib}

\end{document}